\documentclass[11pt]{article}

\usepackage[margin=1in]{geometry}
\usepackage[T1]{fontenc}
\usepackage{lmodern}
\usepackage{amsmath,amssymb,amsthm,mathrsfs}
\usepackage{microtype}
\usepackage{indentfirst}
\usepackage{xcolor}
\usepackage{tikz}
\usetikzlibrary{arrows.meta,matrix,positioning}
\usepackage[colorlinks=true,linkcolor=blue,citecolor=blue,urlcolor=blue]{hyperref}
\hypersetup{
pdftitle={Power-saving density bounds for K-point configurations in Z\string^d}
}

\numberwithin{equation}{section}

\theoremstyle{plain}
\newtheorem{theorem}{Theorem}[section]
\newtheorem{proposition}[theorem]{Proposition}
\newtheorem{lemma}[theorem]{Lemma}
\newtheorem{cor}[theorem]{Corollary}
\theoremstyle{remark}

\newcommand{\R}{\mathbb{R}}

\newcommand{\Z}{\mathbb{Z}}
\newcommand{\N}{\mathbb{N}}

\newcommand{\T}{\mathbb{T}}
\newcommand{\PP}{\mathbb{P}}
\newcommand{\E}{\mathbb{E}}
\newcommand{\C}{\mathbb{C}}

\newcommand{\be}{\beta}

\newcommand{\De}{\Delta}

\newcommand{\si}{\sigma}

\newcommand{\1}{\mathbf{1}}

\newcommand{\subs}{\subseteq}

\newcommand{\abs}[1]{\lvert #1 \rvert}
\newcommand{\norm}[1]{\lVert #1 \rVert}

\newcommand{\rank}{\operatorname{rank}}

\definecolor{crossblue}{HTML}{2F6FAD}
\definecolor{sorange}{HTML}{C66A22}

\tikzset{
	simplexvertex/.style={
		circle,
		draw=black!70,
		fill=white,
		line width=0.5pt,
		minimum size=5.8mm,
		inner sep=0pt,
		font=\small
	},
	svertex/.style={
		simplexvertex,
		fill=sorange,
		draw=sorange!70!black,
		text=white
	},
	tvertex/.style={
		simplexvertex,
		fill=crossblue,
		draw=crossblue!65!black,
		text=white
	},
	crossing/.style={
		draw=black!68,
		line width=1.02pt
	},
	crossinghidden/.style={
		crossing,
		dash pattern=on 2.5pt off 1.6pt,
		opacity=0.72
	},
	tinternal/.style={
		draw=crossblue,
		line width=1.35pt
	},
	sinternal/.style={
		draw=sorange,
		line width=1.45pt
	}
}

\title{Power-saving density bounds for $K$-point configurations in $\mathbb{Z}^d$}
\author{
Andrew Lott \and
Ákos Magyar \and
Nagendar Reddy Ponagandla 
}

\date{\today}

\begin{document}

\maketitle

\begin{abstract}
Let $d,K,N\in \mathbb{N}$ with $K\geq 3$ and $d\geq 4K+4$. Let
$\Delta\subset \Z^d$ be the vertex set of a
nondegenerate $(K-1)$-simplex, and let $A\subseteq[N]^d$ contain no nontrivial 
similar copy of $\Delta$. We prove that
\[
\abs{A}\ll_{\Delta,d}
N^{d-c_{\Delta,d}},
\]   
improving upon a polylogarithmic bound due to Magyar. 
We combine an induction-on-scale argument with
the circle method, and we introduce a ``cut operator'' method to decouple the weighted exponential sum over the system of quadratic forms describing the simplex. Our proof combines ideas from graph theory, functional analysis, and the geometry of numbers.  In the process, we apply Finner's fractional form of H\"older's inequality, the analytic large sieve, and Kim's mean value formula for primitive lattice flags. 
\end{abstract}
\setcounter{tocdepth}{1}
\tableofcontents

\section{Introduction}

A basic problem in additive combinatorics is to determine which finite patterns must occur in dense subsets of integer lattices.  At the qualitative level, the multidimensional Szemerédi theorem of Furstenberg and Katznelson~\cite{FurstenbergKatznelson} gives a general answer: for every finite set $F\subseteq\Z^d$, every subset of $\Z^d$ of positive upper Banach density contains a homothetic copy $x+qF$, with $x\in\Z^d$ and $q\in\N$.  The quantitative bounds obtainable from this general theorem are, however, extremely weak; see \cite{GowersHypergraphRegularity,TaoHypergraphRemoval}.

Here we study the corresponding quantitative problem for similarity classes of simplices.  Let
$\Delta=\{v_0,\ldots,v_{K-1}\}\subseteq\Z^d$
be the vertex set of a fixed nondegenerate $(K-1)$-simplex.  We say that
$\De'=\{x_0,\ldots,x_{K-1}\}\subseteq\Z^d$ is a \emph{similar copy} of
$\Delta$ if there is some $t>0$ such that
$\abs{x_i-x_j}^2=t\abs{v_i-v_j}^2$ for every pair $0\leq i<j\leq K-1$. In other words, the congruence class of the simplex $t\De$ is the determined by its Gram matrix, the $(K-1)\times (K-1)$ matrix with entries $g_{ij}=(v_i-v_0)\cdot (v_j-v_0)$. Writing $X$ for the $d\times(K-1)$ matrix with column vectors $x_i-x_0$ the above quadratic system is equivalent to the matrix equation $X^T X=tG$. Integral solutions $X\in \Z^{d\times (K-1)}$ have been of interests to number theorist, initiated with the pioneering work of Siegel \cite{Siegel1}. Asymptotic formulas were derived with the aid of Siegel modular forms and theta series \cite{Siegel2}; we will make use of Kitaoka's explicit lower bound for the number of integral solutions \cite{Kitaoka}. 

Clearly, $\De'$ is similar to $\De$ if it is a translated and rotated image of a dilate of $\De$. Every homothetic copy is similar, so the qualitative existence of such configurations follows from the preceding theorem; our concern is the quantitative dependence on the density.

Our main result is the following.

\begin{theorem}\label{thm:main}
Let $K\geq3$, and let $\Delta\subseteq\Z^d$ be the vertex set of a fixed nondegenerate integral $(K-1)$-simplex.  If $d\geq4K+4$, then there are
constants $C,c>0$, depending only on $\Delta$ and $d$, such that every set $A\subseteq[N]^d$ containing no nontrivial similar copy of $\Delta$ satisfies $| A|\leq C N^{d-c}.$
\end{theorem}

In the common range $d\geq4K+4$, this improves an earlier result of the second author~\cite{MagyarKPoint}, which states that, when \(d\geq 2K+3\), every set \(A\subseteq[N]^d\) containing no nontrivial similar copy of \(\De\) satisfies
$|A|\leq C_{\De,d}N^d(\log N)^{-2/(11K)}$.
To the best of our knowledge, apart from an earlier version of this paper, this is the only previous quantitative result formulated for similar copies of an arbitrary nondegenerate integral simplex. 

By Chebyshev's bound, Theorem~\ref{thm:main} immediately gives the following quantitative
relative result in the prime lattice; compare this to \cite{CookMagyarTitichetrakun,FoxZhaoPrimeLattice,TaoZieglerPrimeLattice}. Let $\PP$ denote the set of primes.

\begin{cor}
Let $K\geq3$, and let $\Delta\subseteq\Z^d$ be the vertex set
of a fixed nondegenerate integral $(K-1)$-simplex.
If $d\geq4K+4$, then there are constants $C,c>0$,
depending only on $\Delta$ and $d$, such that every set
$A\subseteq(\PP\cap[N])^d$ containing no nontrivial similar
copy of $\Delta$ satisfies
\[
  \abs A\leq C\abs{\PP\cap[N]}^d N^{-c}.
\]
\end{cor}

An equivalent form of our main result is that a set $A\subs [N]^d$ of density $\alpha$ must contain a similar copy once
$N\geq C_{\De,d}\alpha^{-C_{\De,d}}$.
Theorem~\ref{thm:main} also directly implies the following quantitative result in the Euclidean setting.
\begin{cor}\label{cor:continuous}
Let $K\geq3$, $d\geq4K+4$, and let $\Delta\subseteq\Z^d$ be as above.
For every measurable $E\subseteq[0,1]^d$ with $\abs E\geq\alpha>0$, there are $x\in\R^d$, $U\in\mathrm O(d)$, and
$\lambda\geq c_{\Delta,d}\alpha^{C_{\Delta,d}}$ such that $x+\lambda U\Delta\subseteq E$.
\end{cor}
A short proof, based on averaging Theorem~\ref{thm:main} over translates of a sufficiently fine lattice, is given in Appendix~\ref{app:continuous}.  For comparison, a polylogarithmic bound follows from the proof of Bourgain's simplex theorem in dimensions $d\geq K$~\cite{BourgainEuclidean,LyallMagyarWeakRegularity}.  The cut-operator method also has a direct and simpler implementation in the Euclidean setting, which should permit a lower dimension threshold; we do not pursue this direction here.

Beyond~\cite{MagyarKPoint}, quantitative results in the literature concern special configurations or parametrizations by nonsingular matrix progressions.  An important example is the case of corners in the integer lattice: $(x,y),(x+h,y),(x,y+h)$ with $h\ne0$. Building on the seminal work of Kelley and Meka~\cite{KelleyMeka}, Jaber, Liu, Lovett, Ostuni and Sawhney~\cite{JaberEtAlCorners} recently proved that every corner-free set $A\subseteq[N]^2$ satisfies
$|A|\leq N^2\exp\!\bigl(-c(\log N)^{1/600}\bigr)$.
The result also gives the same density bound for every fixed integral triangle.

In the case of avoiding similar copies of a triangle in $\Z^2$, an alternate approach is to observe that complex multiplication gives a parametrization of its similar copies in the form $x,\ x+T_1y,\ x+T_2y$, where $T_1,T_2$, and $T_1-T_2$ are nonsingular rational matrices.  Finding such a copy is equivalent to finding three-term ``matrix progressions''. Pilatte \cite{PilatteRothGeneralized} gives a density bound of the form $(\log N)^{-1-c}$, for sets avoiding these configurations. This approach maybe extended to more complex point configurations via the Hurwitz--Radon construction \cite{GeramitaPullman} at an exponential cost in the dimension. Prendiville's work \cite {PrendivilleMatrix} suggests that longer nonsingular matrix progressions might be handled through higher-order Fourier analysis.  Together with the recent quantitative inverse theorems \cite{LengSahSawhneyInverse,LengSahSawhneySzemeredi} (see also OpenAI~\cite{OpenAIArithmeticProgressions}), this suggests that one might hope for reasonable density bounds. Such a result is not presently available, but even if carried out, this route would give a substantially weaker density bound than Theorem~\ref{thm:main} and would require an exponentially large threshold for the dimension.

Our methods suggest several directions for further work.

The cut operator method may yield quantitative bounds for more complex configurations. A natural test case is a simplex together with its barycenter, or more general barycentric configurations consisting of the vertices of a simplex together with the barycenters of all of its $\ell$-dimensional faces; see Magyar~\cite{MagyarSimplexCenter} for related results in metrics arising from higher-degree forms. Another possibility is a simplex together with an additional point on its circumsphere, including cyclic quadrilaterals in sufficiently high dimension; see Lyall, Magyar and Parshall~\cite{LyallMagyarParshallSpherical} for the corresponding finite-field results.

Theorem 1.1 can also be interpreted as a density theorem for a homogeneous, translation-invariant system of quadratic equations. Indeed, with \(G=(g_{ij})\) denoting the Gram matrix of \(\Delta\), the equations
\[
g_{11}(x_i-x_0)\cdot(x_j-x_0)
-g_{ij}|x_1-x_0|^2=0,
\qquad
1\leq i\leq j\leq K-1,\quad(i,j)\ne(1,1),
\]
characterize similar copies of \(\Delta\), apart from the diagonal solutions \(x_0=\cdots=x_{K-1}\). Our theorem therefore gives a power-saving density bound for sets avoiding nontrivial solutions to this class of quadratic systems. This viewpoint suggests investigating the cut-operator method for other systems
\[
Q_r(x_1-x_0,\ldots,x_{m}-x_0)=0,
\qquad 1\leq r\leq R,
\]
where the \(Q_r\) are homogeneous quadratic forms. Such applications would require suitable moment estimates for the associated cut operators and a positive unweighted solution count. For some related density results concerning translation-invariant quadratic forms in scalar variables, see Keil~\cite{KeilQuadraticForms,KeilQuadraticRefinements} for some related results.

Induction on scale may also improve bounds for problems where strong circle method estimates are available, such as those studied by Doyle and Rice~\cite{DoyleRice,DoyleRicePrimes} for multivariable polynomial differences. A first question is whether replacing the density increment in the classical square-difference argument by induction on scale improves the $O((\log\log N)/\log N)$ density bound of Lyall and Rice~\cite[Theorem~1]{LyallRiceTwoTheorems}.

\section{Acknowledgments and AI disclosure}
This project was initiated while the authors were supported by the
HUN-REN Alfr\'ed R\'enyi Institute of Mathematics (Erd\H{o}s Center)
from January to June 2026. The authors thank the Institute for its
hospitality.

The second author was supported by Simons Foundation grant
MPS-TSM-854813 and grant NKFIH Excellence 154121 from the National Research,
Development and Innovation Office of Hungary.

Large language models were used as a sounding board during the brainstorming phase of the project. They were also used for literature search, LaTeX assistance, and copyediting. However, the final argument is the authors' own, and the manuscript was written and verified by the authors, who take full responsibility for the contents of the proof. 

\section{Notation and standard definitions}

We use the following notation throughout the paper.
\begin{itemize}
\item We write $\mathbb N=\{1,2,\ldots\}$ and, for each positive integer $N$, put $[N]=\{1,\ldots,N\}$. We denote the torus $\R/\Z$ by $\T$. We write $e(x):=e^{2\pi i x}$ and $e_q(x):=e(x/q)$.

\item For $x\in\R$, let
$\norm{x}_{\T}:=\min_{n\in\Z}\abs{x-n}.$
If $x=(x_1,\ldots,x_m)$ lies in $\R^m$, then
$\norm{x}_\infty:=\max_j\abs{x_j}$, while for
$x\in\T^m$ we put
$\norm{x}_\infty:=\max_j\norm{x_j}_{\T}$. The notation $\operatorname{diam}_\infty(E)$
refers to diameter with respect to the maximum norm.

\item For a set $E$, $\1_E$ denotes its indicator function.  If $X$
is finite, $\E_{x\in X}$ denotes normalized averaging over $X$.

\item
For functions $\mu,F:G\to\mathbb C$, define their convolution by
$(\mu*F)(y):=\sum_{x\in G}\mu(x)F(y-x)$.

\item If $T:H_1\to H_2$ is a bounded linear operator between Hilbert spaces,
its adjoint $T^*:H_2\to H_1$ is determined by
$\langle Tf,g\rangle_{H_2} = \langle f,T^*g\rangle_{H_1}.$
We define
\[
\norm{T}_{2\to2}
:=
\sup_{\substack{f\in H_1\\ \norm{f}_{H_1}=1}}
\norm{Tf}_{H_2}.
\]
If $(e_j)_j$ is any orthonormal basis of $H_1$, then
\[
\norm{T}_{\mathrm{HS}}
:=
\left(\sum_j\norm{Te_j}_{H_2}^2\right)^{1/2}.
\]
This value is independent of the choice of orthonormal basis.

\item If $a\in(\Z/q\Z)^m$, then $(a,q)=1$ means that the coordinates
of $a$ together with $q$ have no common divisor greater than $1$.

\item We use $X\ll Y$ and $X=O(Y)$ interchangeably to mean that
$\abs X\leq CY$ for some constant $C>0$. We write $X\gg Y$ if
$Y\ll X$, and $X\asymp Y$ if both $X\ll Y$ and $Y\ll X$.
Subscripts indicate the permitted dependence of the implied constants.

\end{itemize}

\section{Outline of the proof}
The goal of this section is to provide a rough outline of the proof which we hope will aid the reader. 

\subsection{Setup}

Fix a nondegenerate $(K-1)$-simplex $\Delta=\{v_0,\ldots,v_{K-1}\}\subseteq\Z^d$ and set $L:=\binom K2$.
Let $F_0,\ldots,F_{K-1}:[N]^d\to[-1,1]$ and extend them by zero to $\Z^d$. Fix $R=\lfloor\si N\rfloor$ for some sufficiently small constant $0<\si<1$ depending only on $\Delta$. Assume that $N$ is sufficiently large in terms of $\Delta$ and $d$. Set $w=\1_{[1,2)}$ and write $w_R(n)=R^{-2}w(n/R^2)$.

For each integer $n\asymp R^2$, the points $z_0,\ldots,z_{K-1}$ form a similar copy of $\Delta$ with dilation factor $\sqrt n$ precisely when $\abs{z_i-z_j}^2=n\abs{v_i-v_j}^2$ for all $i<j$. We count solutions to these equations with the multilinear expression
\[
  \Lambda_N(F_0,\ldots,F_{K-1})=
  \sum_n w_R(n)
  \sum_{\substack{z_0,\ldots,z_{K-1}\in[N]^d\\
    |z_i-z_j|^2=n|v_i-v_j|^2}}
  \prod_{i=0}^{K-1}F_i(z_i).
\]
Kitaoka's estimate \cite{Kitaoka} supplies
$\Lambda_N(\1_{[N]^d},\ldots,\1_{[N]^d}) \gtrsim N^d R^{(K-1)d-2L}$.

Write $e(t):=\exp(2\pi i t)$. As is customary in the circle method, one expresses the quadratic relations via the integrals
\[
  \1_{|z_i-z_j|^2=n|v_i-v_j|^2}
  =\int_\T e(\be_{ij}|z_i-z_j|^2)
  e(-n\be_{ij}|v_i-v_j|^2)\,d\be_{ij}.
\]
This gives
\begin{equation}\label{Fouriercount:intro}
  \Lambda_N(F_0,\ldots,F_{K-1})
  =\int_{\T^L}
  \widehat w_R\big(\sum_{i<j}|v_i-v_j|^2\beta_{ij}\big)\,
  \mathcal S_\beta(F_0,\ldots,F_{K-1})\,d\beta.
\end{equation}
Here $\widehat w_R(t):=\sum_n w_R(n)e(-nt)$, so $\|\widehat w_R\|_{L^\infty(\T)}\lesssim1$ and 
\begin{equation}\label{eq:intro-Sbeta}
  \mathcal S_\beta(F_0,\ldots,F_{K-1})
  =\sum_{z_0,\ldots,z_{K-1}\in[N]^d}
  e\!\left(\sum_{i<j}\beta_{ij}|z_i-z_j|^2\right)
  \prod_{i=0}^{K-1}F_i(z_i).
\end{equation}
The key difficulty in the argument is proving sufficiently strong estimates for the exponential sum $\mathcal S_\beta(F_0,\ldots,F_{K-1})$ that retain the information given by the factors $F_i$. 

\subsection{The induction on scale}\label{sec:intro-induction}

Our induction-on-scale argument is based on the ``dichotomy between structure
and randomness'' found in the density increment arguments of Roth \cite{Roth1953} and
S\'ark\"ozy \cite{Sarkozy1978}, combined with the averaging argument of 
Lyall and Magyar~\cite[\S3.1, p.~694]{LyallMagyarDistanceGraphs}.
It is also interesting to compare our proof to the inductive formulation of the density-increment argument in
Green--Sawhney~\cite[\S6]{GreenSawhney} and the induction scheme
in the OpenAI manuscript on square differences~\cite[\S9]{OpenAISquareDifferences}.

Fix a sufficiently small $c>0$ and a sufficiently large $M\geq1$.
Suppose that there is a set $A\subseteq[N]^d$ containing no nontrivial
similar copy of $\Delta$ with $|A|>MN^{d-c}$, and choose an example for which
$N$ is minimal. Write $\alpha=|A|/N^d$. Since $\alpha\leq1$, we have
$N>M^{1/c}$, so choosing $M$ large ensures that $N$ exceeds every
threshold required below. Our goal is to show that
\begin{equation}\label{eq:intro-counting-comparison}
 \left|\Lambda_N(\1_A,\ldots,\1_A)
 -\alpha^K\Lambda_N(\1_{[N]^d},\ldots,\1_{[N]^d})\right|
 \leq\varepsilon\alpha^KN^{Kd-2L},
\end{equation}
where $\varepsilon>0$ is sufficiently small depending only on
$\Delta,d$. Since $R\asymp_\Delta N$, Kitaoka's lower bound then
implies that the weighted number $\Lambda_N(\1_A,\ldots,\1_A)$
of similar copies of $\Delta$ in $A$ is positive, a contradiction.

For $q\geq1$ and $1\leq\ell<N$, let $\mu_{q,\ell}$ be normalized
counting measure on $q\{0,\ldots,\ell-1\}^d$, and let
$\alpha_x=(\mu_{q,\ell}*\1_A)(x)$ be the density of $A$ in the translated
$q$-spaced subcube $x-q\{0,\ldots,\ell-1\}^d$. (We use the term \emph{subcube} in order to emphasize the fact that this is a higher-dimensional analogue of the \emph{subprogressions} used in the density increment arguments over the integers.)
Translating and dividing by $q$ preserves lack of similar copies of $\Delta$,
so minimality of $N$ gives
\[
 0\leq\alpha_x\leq M\ell^{-c}<\alpha(N/\ell)^c.
\]
Averaging preserves the total mass, so $\sum_{x\in\Z^d}\alpha_x=\alpha N^d$.
Combining this identity with the upper bound, for $q\ell\leq N$ we obtain
\begin{equation}\label{eq:intro-local-density-deviation}
 \frac{1}{N^d}\sum_{x\in[N+q\ell]^d}(\alpha_x-\alpha)^2
 \ll_d\alpha^2\bigl((N/\ell)^c-1+q\ell/N\bigr).
\end{equation}
When $(N/\ell)^c-1$ and $q\ell/N$ are small, this bound says that
the density of $A$ in $x-q\{0,\ldots,\ell-1\}^d$ is close to
$\alpha$ for most $x\in[N+q\ell]^d$. The term $q\ell/N$ accounts
for the boundary case (subcubes which are not contained in $[N]^d$). In other words, if we view $\alpha_x$ as a random variable giving the local density of $A$ on each $q$-spaced subcube $x-q\{0,\ldots,\ell-1\}^d$, then its mean is roughly $\alpha$ and its variance is small.

Then the density increment argument of Roth and Sárközy in the contrapositive would suggest that $A$ has many similar copies of $\Delta$, as would follow from \eqref{eq:intro-counting-comparison} combined with Kitaoka's lower bound. In the spirit of Sárközy's circle method argument for bounding sets lacking square differences, we use our cut operator method to derive sufficiently strong estimates for the exponential sum $\mathcal S_{\beta}(F_0,\ldots,F_{K-1})$, which in turn yield \eqref{eq:intro-counting-comparison}.

\subsection{The cut-operator reduction}

Write $\mathcal E$ for the edge set of the complete graph on $K$ vertices, and note that $L=\abs{\mathcal E}=\binom K2$. We use the convention $\beta_{ji}=\beta_{ij}$ for $i<j$.

Fix a balanced cut $S\mid T$, that is, a partition of the vertex set into two classes whose sizes differ by at most one. The phase $\mathcal Q_\be(z)$ in
\eqref{eq:intro-Sbeta} decomposes as
\[
  \mathcal Q_\beta(z)
  =\mathcal Q_\beta^S(z_S)+\mathcal Q_\beta^T(z_T)
  -2\sum_{i\in S,\,j\in T}\beta_{ij}z_i\cdot z_j,
\]
where $z_S=(z_i)_{i\in S}$ and $z_T=(z_j)_{j\in T}$.
The first two terms depend on only one side of the cut. For fixed $\Theta\in\mathbb{T}^{|S||T|}$, the crossing term naturally leads to the operator
\[
  \mathcal K_{R,d}^{S,T}(\Theta):
  \ell^2(([R]^d)^T)\to\ell^2(([R]^d)^S)
\]
defined by
\[
  \bigl(\mathcal K_{R,d}^{S,T}(\Theta)h\bigr)(z_S)
  :=
  \sum_{z_T\in([R]^d)^T}
  e\!\left(\sum_{i\in S,\,j\in T}\Theta_{ij}z_i\cdot z_j\right)h(z_T).
\]
We call $\mathcal K_{R,d}^{S,T}(\Theta)$ a \emph{cut operator}. Because the dot product separates coordinatewise, $\mathcal K_{R,d}^{S,T}(\Theta)$ is the $d$-fold tensor product of the operator $\mathcal K_R^{S,T}(\Theta):=\mathcal K_{R,1}^{S,T}(\Theta)$. For ease of notation, we work with the normalized operator norm
$m_R^{S,T}(\Theta):=R^{-K/2}\norm{\mathcal K_R^{S,T}(\Theta)}_{2\to2}$.
This normalization gives $0\leq m_R^{S,T}(\Theta)\leq1$ for every $\Theta\in\mathbb T^{|S||T|}$.
Let $\beta_{S,T}=(\be_{ij})_{i\in S,j\in T}$ denote the frequencies corresponding to the edges crossing the cut $S\mid T$. A key observation is to write
\[
  \mathcal S_\beta(F_0,\ldots,F_{K-1})
  =
  \big\langle
  U_S\mathcal K_{N,d}^{S,T}(-2\beta_{S,T})U_TH_T,
  H_S
  \big\rangle,
\]
where $H_S(z_S):=\prod_{i\in S}F_i(z_i)$ and $H_T(z_T):=\prod_{j\in T}F_j(z_j)$.
The unitary multiplication operators $U_S,U_T$ are given by
\[
  (U_Sh)(z_S):=e(\mathcal Q_\beta^S(z_S))h(z_S),
  \qquad
  (U_Th)(z_T):=e(\mathcal Q_\beta^T(z_T))h(z_T).
\]
We have $\|H_S\|_{\ell^2(([N]^d)^S)}=\prod_{i\in S}\|F_i\|_{\ell^2(\Z^d)}$,
and likewise for $H_T$. The Cauchy--Schwarz inequality, unitarity, and the
multiplicativity of the $\ell^2\to\ell^2$ operator norm
with respect to tensor products imply that
\begin{equation}\label{eq:intro-one-cut-reduction}
  \abs{\mathcal S_\beta(F_0,\ldots,F_{K-1})}
  \leq
  N^{Kd/2}\,m_N^{S,T}(-2\beta_{S,T})^d
  \prod_{i=0}^{K-1}\norm{F_i}_{\ell^2(\Z^d)}.
\end{equation}
The problem of estimating the exponential sum $\mathcal S_\beta(F_0,\ldots,F_{K-1})$ has thus been reduced to understanding the normalized one-dimensional cut-operator norms.  Similar reductions can be found in the work of Greenleaf--Iosevich--Taylor~\cite{GreenleafIosevichTaylor,GreenleafIosevichTaylorMicrolocal}, Liu \cite{LiuPrimeQuadrics}, Cook--Magyar \cite{CookMagyarDiophantine}, and Green \cite{GreenQuadraticForms}.

We consider all balanced cuts of the vertex set $\Delta$. 
Let $\mathscr B_K$ be the set of balanced cuts; we introduce the notation $w_K:=L/(|S||T||\mathscr B_K|)$ and $\tau:=d|S||T|/(2L).$ 

By symmetry, for every edge $e\in\mathcal E$ we have
\begin{equation}\label{eq:intro-fractional-cover}
  w_K\#\{S\mid T\in\mathscr B_K:e\text{ crosses }S\mid T\}=1.
\end{equation}
Each one-dimensional norm occurs to the power $d$ in
\eqref{eq:intro-one-cut-reduction}, so its exponent in the geometric mean
over all balanced cuts is $d/|\mathscr B_K|=2\tau w_K$.
For $\theta\in\T^L$, denote this geometric mean by
\begin{equation}\label{eq:balanced-cut-mean}
  I_N(\theta):=
  \prod_{S\mid T\in\mathscr B_K}
  m_N^{S,T}(-2\theta_{S,T})^{2\tau w_K}.
\end{equation}
Taking the geometric mean of
\eqref{eq:intro-one-cut-reduction} over the balanced cuts gives
\begin{equation}\label{eq:intro-geometric-cut-bound}
  \abs{\mathcal S_\beta(F_0,\ldots,F_{K-1})}
  \leq
  N^{Kd/2}I_N(\beta)
  \prod_{i=0}^{K-1}\norm{F_i}_{\ell^2(\Z^d)}.
\end{equation}

\subsection{The mean value estimate}

We first give a mean value estimate for the cut operator norm.

\begin{lemma}\label{lem:rectangular-moment}
Let $S\mid T$ be a balanced cut. For every positive integer $R$,
\begin{equation}\label{eq:rectangular-moment}
  \int_{\T^{\abs S\abs T}}
  m_R^{S,T}(\Theta)^{2K+4}\,d\Theta
  \ll_K R^{-2\abs S\abs T}.
\end{equation}
\end{lemma}

The power $R^{-2\abs S\abs T}$ is the natural saving associated with the $\abs S\abs T$ crossing frequencies. To prove the lemma, we observe first that
\[
  \|\mathcal K_R^{S,T}(\Theta)\|_{2\to2}^{2K+4}
  =
  \|(\mathcal K_R^{S,T}(\Theta)
  \mathcal K_R^{S,T}(\Theta)^*)^{K+2}\|_{2\to2}
\]
and that the latter norm is bounded by the trace of
$(\mathcal K_R^{S,T}(\Theta)\mathcal K_R^{S,T}(\Theta)^*)^{K+2}$.
Expanding the trace and integrating in $\Theta$ reduces the problem, up to at most $R^K$ choices of base points, to counting bounded integer matrices $U\in\Z^{\abs S\times(K+1)}$ and $V\in\Z^{\abs T\times(K+1)}$ satisfying
$UV^{\mathsf T}=0$.

We use a geometry of numbers argument to exploit special structure of this matrix equation. In particular, we consider the saturated row lattices $\Gamma_U$ and $\Gamma_V$ given by $U$ and $V$. When $U$ and $V$ are nonzero, these lattices give the primitive $2$-step lattice flag
$\Gamma_U\subseteq\Gamma_V^{\perp_{\Z}}\subseteq\Z^{K+1}$, where $\Gamma_V^{\perp_{\Z}}$ is the integral orthogonal complement of $\Gamma_V$. In this way, instead of counting solutions to the matrix equation directly, we may count primitive $2$-step lattice flags subject to certain constraints. Geometry of numbers and Kim's mean-value formula for primitive lattice flags~\cite{KimFlags} then give the upper bound
\[
  \#\{(U,V):UV^{\mathsf T}=0,\ \norm U_\infty,\norm V_\infty\ll R\}
  \ll_K R^{K(K+1)-2\abs S\abs T}.
\]
This gives precisely
Lemma~\ref{lem:rectangular-moment} after normalization by $R^{-K(K+2)}$.

The lattice approach can be replaced by a classical circle-method approach, but this would yield a worse dimension threshold. Indeed, since $UV^{\mathsf T}=0$ is a system of bilinear forms, applying the circle-method estimates in~\cite{RydinMyerson,HochfilzerBihomogeneous} would lead to a dimension threshold $d\gg K^2$. We also do not know of a way of making Kitaoka's method handle our particular matrix equation.

\subsection{The minor arc estimate}

Fix $2\leq Q\leq R$. For $e=\{i,j\}\in\mathcal E$, write $\beta_e=\beta_{ij}$.
Define the major arcs by
\[
 \mathfrak M_Q:=\left\{\beta\in\T^L:
 \|q\beta\|_\infty\leq Q/R^2
 \text{ for some integer }1\leq q\leq Q\right\},
\]
and let $\mathfrak m_Q:=\T^L\setminus\mathfrak M_Q$ be the minor arcs. 

For $\xi\in\T$, let
$\mathcal K_P(\xi)=\bigl(e(\xi rs)\bigr)_{r,s\in[P]}$ and write
$m_P(\xi)=P^{-1}\norm{\mathcal K_P(\xi)}_{2\to2}$ for its normalized operator norm. We view $\mathcal K_P(\xi)$ as a ``one-dimensional cut operator''. 
For every crossing edge $e=(i,j)$ with $i\in S$ and $j\in T$, some linear algebra along with an application of Cauchy--Schwarz yields the bound
$m_P^{S,T}(\Theta)\leq m_P(\Theta_{ij})$, so that the cut operator norm is controlled by the one-dimensional cut operator norm associated to any fixed crossing edge. The next lemma shows that for each $\beta$ in the minor arcs there is an edge on which we obtain cancellation. This comes from a simple application of the analytic large sieve. 

\begin{lemma}\label{lem:scalar-minor}
Let $2\leq Q\leq R\leq P$, with $P$ a positive integer. For every $\beta\in\mathfrak m_Q$
there is an edge $e\in\mathcal E$ such that $m_P(-2\beta_e)^2\ll Q^{-1/L}$.
\end{lemma}

By \eqref{eq:intro-geometric-cut-bound}, it suffices to estimate $\int_{\mathfrak m_Q}I_N(\beta)\,d\beta$.
Put $\gamma:=\min\{1,\tau-K-2\}$, which is positive since $d\geq4K+4$.
Fix $e\in\mathcal E$. For a balanced cut $S\mid T$, let
$\beta_{S,T}=(\beta_{ij})_{i\in S,j\in T}$ denote the frequencies corresponding to the crossing edges.
(The edge $e$ crosses $S\mid T$ if it has one endpoint in each class.)
By \eqref{eq:intro-fractional-cover}, the weights of the cuts which $e$ crosses
sum to one. Using the definitions of $\gamma$ and $I_P(\beta)$ along with the bound $m_P^{S,T}(-2\beta_{S,T})\leq m_P(-2\beta_e)$ for cuts crossed by $e$, we have
\begin{equation}\label{eq:intro-edge-extraction}
 I_P(\beta)\leq m_P(-2\beta_e)^{2\gamma}
 \prod_{S\mid T\in\mathscr B_K}m_P^{S,T}(-2\beta_{S,T})^{(2K+4)w_K}.
\end{equation}
Applying Finner's fractional form of H\"older's inequality~\cite{Finner}
with weight $w_K$ for each balanced cut and then Lemma~\ref{lem:rectangular-moment} gives
\begin{align}
 \int_{\T^L}\prod_{S\mid T\in\mathscr B_K}
 m_P^{S,T}(-2\beta_{S,T})^{(2K+4)w_K}\,d\beta
 &\leq\prod_{S\mid T\in\mathscr B_K}
 \left(\int_{\T^{|S||T|}}m_P^{S,T}(\Theta)^{2K+4}\,d\Theta\right)^{w_K}\notag\\
 &\ll_K\prod_{S\mid T\in\mathscr B_K}P^{-2|S||T|w_K}
 =P^{-2L}.
 \label{eq:intro-full-cut-moment}
\end{align}
Here multiplication by $-2$ preserves Haar measure, and $\sum_{S\mid T\in\mathscr B_K}|S||T|w_K=L$.
Taking $P=N$ and using the edge supplied by Lemma~\ref{lem:scalar-minor} for each
$\beta\in\mathfrak m_Q$, we obtain
\begin{equation}\label{eq:intro-minor-cut-integral}
 \int_{\mathfrak m_Q}I_N(\beta)\,d\beta\ll_{K,d}Q^{-\gamma/L}N^{-2L}.
\end{equation}

\textbf{Example:} For $K=4$, the balanced cuts are $\{0,1\}\mid\{2,3\}$,
$\{0,2\}\mid\{1,3\}$, and $\{0,3\}\mid\{1,2\}$. Each cut has four
crossing edges. The edge $\{0,1\}$ crosses the last two cuts. Suppose
$m_N(-2\beta_{01})^2\ll Q^{-1/6}$, as in Lemma~\ref{lem:scalar-minor} with $L=6$.
Since $w_4=1/2$ and the moment exponent in Lemma~\ref{lem:rectangular-moment} is $12$,
extracting $m_N(-2\beta_{01})^{2\gamma}$ before applying Finner gives
\[
 Q^{-\gamma/6}(N^{-8})^{1/2}(N^{-8})^{1/2}(N^{-8})^{1/2}
 =Q^{-\gamma/6}N^{-12}.
\]

\subsection{The major arc estimate}
On a major arc write $\beta=a/q+\theta$ with $(a,q)=1$. By definition,
$1\leq q\leq Q$ and $\|\theta\|_\infty\leq Q/(qR^2)$.
For $1\leq\ell<N$ with $q\ell\leq N$, we average each input
using the measure $\mu_{q,\ell}$ from Section~\ref{sec:intro-induction}.
For $u_i\in\Z^d$ and $r_i\in(\Z/q\Z)^d$, let $y_i$ be the unique point
in $u_i-\{0,\ldots,q-1\}^d$ congruent to $r_i$ modulo $q$.
We approximate $\mathcal S_{a/q+\theta}(F_0,\ldots,F_{K-1})$ by the major arc form
\begin{align*}
 \mathcal M_{q,a,\theta}^{(\ell)}(F_0,\ldots,F_{K-1})=\E_{r\bmod q}\sum_{u\in(\Z^d)^K}
 e_q\bigl(\mathcal Q_a(r)\bigr)e\bigl(\mathcal Q_\theta(u)\bigr)
 \prod_{i=0}^{K-1}(\mu_{q,\ell}*F_i)(y_i),
\end{align*}
where the average is over $r=(r_0,\ldots,r_{K-1})\in((\Z/q\Z)^d)^K$.
The error from this approximation satisfies
\begin{align*}
 \left|\mathcal S_{a/q+\theta}(F_0,\ldots,F_{K-1})
 -\mathcal M_{q,a,\theta}^{(\ell)}(F_0,\ldots,F_{K-1})\right|\ll_{K,d}N\|\theta\|_\infty q\ell
 \prod_{i=0}^{K-1}\|F_i\|_{\ell^1(\Z^d)}.
\end{align*}
We will choose $\ell$ as a function of $q$ and $\theta$ to control this error. Each $y_i$ occurs for exactly one residue $r_i$ and $q^d$ values
of $u_i$, so
\[
 \E_{r_i\bmod q}\sum_{u_i\in\Z^d}|(\mu_{q,\ell}*F_i)(y_i)|^2
 =\|\mu_{q,\ell}*F_i\|_{\ell^2(\Z^d)}^2.
\]

With normalized counting measure on both spaces, define an operator
\[
  \mathcal K^{S,T}_{q,a,d}:
  \ell^2\!\left(((\Z/q\Z)^d)^T\right)
  \longrightarrow
  \ell^2\!\left(((\Z/q\Z)^d)^S\right)
\]
given by
\[
  \bigl(\mathcal K_{q,a,d}^{S,T}H\bigr)(r_S)
  :=
  \E_{r_T\bmod q}
  e_q\!\left(-2\sum_{i\in S,\,j\in T}
  a_{ij}\,r_i\cdot r_j\right)H(r_T).
\]
(We call this a complete cut operator.)
By the cut operator reduction, for each cut $S\mid T$, the major arc form is controlled by
$\mathcal K^{S,T}_{q,a,d}\otimes\mathcal K_{P,d}^{S,T}(-2\theta_{S,T})$,
where we take $P=2N$ since $N+q\ell\leq 2N$. Each factor $(\mu_{q,\ell}*F_i)(y_i)$ depends on both $u_i$ and $r_i$ through $y_i$. The cut operator is a tensor product, so its operator norm factors as
\[
  \|\mathcal K_{q,a,d}^{S,T}
  \otimes\mathcal K_{P,d}^{S,T}(-2\theta_{S,T})\|_{2\to2}
  =
  \|\mathcal K_{q,a,d}^{S,T}\|_{2\to2}
  \|\mathcal K_{P,d}^{S,T}(-2\theta_{S,T})\|_{2\to2}.
\]
We therefore need bounds for the operator norms of
$\mathcal K^{S,T}_{q,a,d}$ and $\mathcal K_{P,d}^{S,T}(-2\theta_{S,T})$.

An elementary orthogonality argument gives
$\norm{\mathcal K_{q,a,d}^{S,T}}_{2\to2} =\abs{\operatorname{im}(2a_{S,T})}^{-d/2}$,
where the matrix $2a_{S,T}=(2a_{ij})_{i\in S,j\in T}$ is regarded as a homomorphism $(\Z/q\Z)^T\to(\Z/q\Z)^S$.
The geometric mean of the complete cut operator norms is
$\mathfrak s_q(a):=\prod_{S\mid T\in\mathscr B_K}
|\operatorname{im}(2a_{S,T})|^{-\tau w_K}$. 

The tensor product identity gives $\|\mathcal K_{P,d}^{S,T}(\Theta)\|_{2\to2}=P^{Kd/2}m_P^{S,T}(\Theta)^d$.
Taking the geometric mean over the balanced cuts at $\Theta=-2\theta_{S,T}$ gives
$P^{Kd/2}I_P(\theta)$ by \eqref{eq:balanced-cut-mean}. The factor $I_P(\theta)$ plays the role of an upper bound
for the absolute value of the oscillatory integral in the classical major arc estimate.

Taking $P=2N$ and adding the approximation error gives the pointwise major arc
estimate of Proposition~\ref{prop:smoothed-major-bound},
\begin{align}
 \abs{\mathcal S_{a/q+\theta}(F_0,\ldots,F_{K-1})}
 \ll_{K,d} N^{Kd/2}\mathfrak s_q(a)I_{2N}(\theta)
 \prod_{i=0}^{K-1}\norm{\mu_{q,\ell}*F_i}_{\ell^2(\Z^d)}+N\|\theta\|_\infty q\ell
 \prod_{i=0}^{K-1}\norm{F_i}_{\ell^1(\Z^d)}.
 \label{eq:intro-smoothed-major-bound}
\end{align}
For the input $F_i=\1_A-\alpha\1_{[N]^d}$, the convolution
$(\mu_{q,\ell}*F_i)(x)$ equals $\alpha_x-\alpha$ on subcubes contained
in $[N]^d$. Thus \eqref{eq:intro-local-density-deviation}
controls $N^{-d}\|\mu_{q,\ell}*F_i\|_{\ell^2(\Z^d)}^2$ once we account
for the subcubes which are not contained in $[N]^d$.

Since $|\widehat w_R|\leq1$, the major arc contribution to
\eqref{Fouriercount:intro} is bounded in absolute value by
\[
 \int_{\mathfrak M_Q}
 \abs{\mathcal S_\beta(F_0,\ldots,F_{K-1})}\,d\beta
 \leq
 \sum_{1\leq q\leq Q}\sum_{\substack{a\bmod q\\(a,q)=1}}
 \int_{\|\theta\|_\infty\leq Q/(qR^2)}
 \abs{\mathcal S_{a/q+\theta}(F_0,\ldots,F_{K-1})}\,d\theta.
\]
The pointwise bound on the integrand on the right is
\begin{align*}
 \abs{\mathcal S_{a/q+\theta}(F_0,\ldots,F_{K-1})}
 &\ll_{K,d} N^{Kd/2}\mathfrak s_q(a)I_{2N}(\theta)
 \prod_{i=0}^{K-1}\norm{\mu_{q,\ell}*F_i}_{\ell^2(\Z^d)}{}+N\|\theta\|_\infty q\ell
 \prod_{i=0}^{K-1}\norm{F_i}_{\ell^1(\Z^d)}.
\end{align*}
As $q(1+N^2\|\theta\|_\infty)$ grows, we take $\ell$ smaller to control
the approximation error. This increases the factor $(N/\ell)^c$ on the
right-hand side of \eqref{eq:intro-local-density-deviation}. To absorb this loss, we prove the estimate
\begin{equation}\label{eq:intro-weighted-major-moment}
 \left(\sum_{q\geq1}q^\gamma
 \sum_{\substack{a\bmod q\\(a,q)=1}}\mathfrak s_q(a)\right)
 \int_{\|\theta\|_\infty\leq1/(8P)}
 (1+P^2\|\theta\|_\infty)^\gamma I_P(\theta)\,d\theta
 \ll_{K,d}P^{-2L}.
\end{equation}
We apply this with $P=2N$ and $Q/R^2\leq1/(8P)$.

To estimate the integral in \eqref{eq:intro-weighted-major-moment}, fix
$\theta$ with $\|\theta\|_\infty\leq1/(8P)$ and choose an edge $e$ with
$|\theta_e|=\|\theta\|_\infty$. By \eqref{eq:intro-fractional-cover},
the balanced cuts crossed by $e$ have total weight one. For each such cut,
$m_P^{S,T}(-2\theta_{S,T})\leq m_P(-2\theta_e)$, and the analytic large sieve gives
\[
 \prod_{\substack{S\mid T\in\mathscr B_K\\e\ \text{crosses }S\mid T}}
 m_P^{S,T}(-2\theta_{S,T})^{2\gamma w_K}
 \leq m_P(-2\theta_e)^{2\gamma}
 \ll_{K,d}(1+P^2\|\theta\|_\infty)^{-\gamma}.
\]
We use \eqref{eq:intro-edge-extraction} at $\beta=\theta$, enlarge the domain
on the right to $\T^L$, and apply \eqref{eq:intro-full-cut-moment} to obtain
\begin{align*}
 \int_{\|\theta\|_\infty\leq1/(8P)}
 (1+P^2\|\theta\|_\infty)^\gamma I_P(\theta)\,d\theta
 &\ll_{K,d}\int_{\T^L}\prod_{S\mid T\in\mathscr B_K}
 m_P^{S,T}(-2\theta_{S,T})^{(2K+4)w_K}\,d\theta\\
 &\ll_{K,d}P^{-2L}.
\end{align*}
For the factor $\mathfrak{s}_q(a)$ in \eqref{eq:intro-weighted-major-moment},
Lemma~\ref{lem:modular-numerator-sum} bounds the sum over $a$ by
$O_{K,d}(q^{-1-\eta})$ with $\eta>1$. Since $\gamma\leq1$, the sum over
$q$ converges. This proves \eqref{eq:intro-weighted-major-moment}
(Lemma~\ref{lem:weighted-major-moment}).

\subsection{Completing the argument}

Recall that the goal is to compare 
$\Lambda_N(\1_A,\ldots,\1_A)$ and
$\alpha^K\Lambda_N(\1_{[N]^d},\ldots,\1_{[N]^d})$.
Put $g=\1_A-\alpha\1_{[N]^d}$. By telescoping, the difference between $\Lambda_N(\1_A,\ldots,\1_A)$ and $\alpha^K\Lambda_N(\1_{[N]^d},\ldots,\1_{[N]^d})$ is a sum of terms of the form 
\[
\Lambda_N\bigl(\underbrace{\1_A,\ldots,\1_A}_{j},g,\underbrace{\alpha\1_{[N]^d},\ldots,\alpha\1_{[N]^d}}_{K-1-j}\bigr).
\]
Fix $j$ and denote the inputs by $F_0,\ldots,F_{K-1}$. By \eqref{Fouriercount:intro},
\[
 \Lambda_N(F_0,\ldots,F_{K-1})
 =\int_{\T^L}\widehat w_R\!\left(\sum_{r<s}|v_r-v_s|^2\beta_{rs}\right)
 \mathcal S_\beta(F_0,\ldots,F_{K-1})\,d\beta.
\]
Take $Q=N^\kappa$ with $0<\kappa<1/(L+3)$.
By \eqref{eq:intro-geometric-cut-bound} and \eqref{eq:intro-minor-cut-integral},
the minor arcs contribute $O_{\Delta,d}(Q^{-\gamma/L}\alpha^{K/2}N^{Kd-2L})$,
which is $o(\alpha^KN^{Kd-2L})$ for $c$ sufficiently small since $\alpha>N^{-c}$.
As $|\widehat w_R|\leq1$, it remains to show that
\begin{equation}\label{majorarcgoal}
 \int_{\mathfrak M_Q}\abs{\mathcal S_\beta(F_0,\ldots,F_{K-1})}\,d\beta
 \leq \frac{\varepsilon}{2K}\alpha^KN^{Kd-2L}.
\end{equation}

Fix $\beta=a/q+\theta\in\mathfrak M_Q$. With $t=q(1+N^2\|\theta\|_\infty)$,
choose $\ell=\lfloor cN/t^{L+3}\rfloor$. The major arc estimate
\eqref{eq:intro-smoothed-major-bound} then gives
\begin{align*}
 \abs{\mathcal S_{a/q+\theta}(F_0,\ldots,F_{K-1})}
 &\ll_{K,d} N^{Kd/2}\mathfrak s_q(a)I_{2N}(\theta)
 \prod_{i=0}^{K-1}\|\mu_{q,\ell}*F_i\|_2+c\,\alpha^KN^{Kd}t^{-L-2}.
\end{align*}
Recall the variance bound \eqref{eq:intro-local-density-deviation} for the
densities of $A$ on $q$-spaced subcubes. An averaging argument gives
\[
 \|\mu_{q,\ell}*g\|_2^2
 \ll_d \alpha^2N^d\bigl((N/\ell)^c-1+q\ell/N\bigr).
\]
For $i\ne j$, we have
$\|\mu_{q,\ell}*F_i\|_2\leq\alpha N^{d/2}(N/\ell)^{c/2}$ (a consequence of the fact that in this case $F_i$ is either $\1_A$ or $\alpha\1_{[N]^d}$).
Splitting according to the size of $t$ and applying \eqref{eq:intro-weighted-major-moment} gives \eqref{majorarcgoal}.

\section{Kitaoka's estimate and the orthogonality relation}

Recall the definitions of $\Lambda_N$ and $w_R$ from the outline. Let $D_\Delta:=\max_{\{i,j\}\in\mathcal E}\abs{v_i-v_j}$ be the
diameter of $\Delta$, and put $D_0:=1+\lfloor\sqrt2\,D_\Delta\rfloor$.
When $w_R(n)\neq0$, the simplex equations give
$\abs{z_i-z_j}=\sqrt n\,\abs{v_i-v_j}<\sqrt2 R D_\Delta<D_0R$
for every pair $0\leq i<j\leq K-1$.

\begin{lemma}[Kitaoka's simplex count]\label{Kitaoka}
If $d\geq4K+4$ and $N\geq N_0$, then
\begin{equation}\label{eq:count-normalization}
  \Lambda_N(\1_{[N]^d},\ldots,\1_{[N]^d})
  \gg_{\Delta,d}N^dR^{(K-1)d-2L}.
\end{equation}
\end{lemma}

\begin{proof}
Put $m_0=0$, and let
\[
\mathcal I_N
:=
\bigl\{
x\in[N]^d:
1+D_0R\leq x_j\leq N-D_0R
\text{ for }1\leq j\leq d
\bigr\}.
\]
Since $D_0R<N/4$, we have $\abs{\mathcal I_N}\gg_dN^d$.

Fix $x\in\mathcal I_N$, and write $z_i=x+m_i$. Since
$z_i-z_j=m_i-m_j,$
the configuration equations are
$\abs{m_i-m_j}^2=n\abs{v_i-v_j}^2$
for every edge $\{i,j\}\in\mathcal E$.
Because $m_0=0$, simple rearranging shows that these equations are
equivalent to
$m_i\cdot m_j = n\,(v_i-v_0)\cdot(v_j-v_0)$
for every pair of indices $1\leq i,j\leq K-1$.
For $R^2\leq n<2R^2$, Kitaoka's estimate
\cite[Theorems~A--C]{Kitaoka} in the form recorded in
\cite[(2.16) and (3.10)]{MagyarKPoint} gives
\[
\#\left\{
\begin{aligned}
  &(m_1,\ldots,m_{K-1})\in(\Z^d)^{K-1}:\\[-2pt]
  &\abs{m_i-m_j}^2=n\abs{v_i-v_j}^2
  \quad\text{for every }\{i,j\}\in\mathcal E
\end{aligned}
\right\}
\gg_{\Delta,d}
n^{(K-1)(d-K)/2}.
\]
Since $\abs{m_i}=\sqrt n\,\abs{v_i-v_0}<D_0R$, we have $x+m_i\in[N]^d$ for every
$x\in\mathcal I_N$. There are $R^2$ integers $n$ in this range, and
$w_R(n)=R^{-2}$ for each of them. Hence
\[
\begin{aligned}
  \Lambda_N(\1_{[N]^d},\ldots,\1_{[N]^d})
  \gg_{\Delta,d}
  N^d\cdot R^2\cdot R^{-2}
  \cdot R^{(K-1)(d-K)}
  =N^dR^{(K-1)d-2L}.
\end{aligned}
\]
\end{proof}

\subsection{Applying the orthogonality relation}

For
$\beta=(\beta_{ij})_{\{i,j\}\in\mathcal E}\in\T^L,$
we use the convention $\beta_{ji}=\beta_{ij}$.  For a tuple
$u=(u_0,\ldots,u_{K-1})$, set
$\mathcal Q_\beta(u) := \sum_{\{i,j\}\in\mathcal E} \beta_{ij}\abs{u_i-u_j}^2$.

We apply the orthogonality relation to detect whether
$\abs{z_i-z_j}^2-n\abs{v_i-v_j}^2=0$ for every edge
$\{i,j\}\in\mathcal E$ in the definition of $\Lambda_N$.
If $z=(z_0,\ldots,z_{K-1})$ and
$v=(v_0,\ldots,v_{K-1})$, the resulting phase is
\[
\sum_{\{i,j\}\in\mathcal E}
\beta_{ij}
\left(
\abs{z_i-z_j}^2-n\abs{v_i-v_j}^2
\right)
=
\mathcal Q_\beta(z)-n\mathcal Q_\beta(v).
\]
Thus
\[
\prod_{\{i,j\}\in\mathcal E}
\1_{\abs{z_i-z_j}^2=n\abs{v_i-v_j}^2}
=
\int_{\T^L}
e\bigl(\mathcal Q_\beta(z)-n\mathcal Q_\beta(v)\bigr)\,d\beta.
\]
Substituting this identity into the counting form gives
\begin{equation}\label{orthorep}
  \Lambda_N(F_0,\ldots,F_{K-1})
  =
  \int_{\T^L}
  \sum_{n\in\Z}w_R(n)
  e\bigl(-n\mathcal Q_\beta(v)\bigr)
  \sum_{z_0,\ldots,z_{K-1}\in\Z^d}
  e\bigl(\mathcal Q_\beta(z)\bigr)
  \prod_{i=0}^{K-1}F_i(z_i)\,d\beta.
\end{equation}

Observe that the sum over $n$ in \eqref{orthorep} is  $\widehat w_R\bigl(\mathcal Q_\beta(v)\bigr)$.
For $\beta\in\T^L$, write
\begin{equation}\label{eq:fixed-parameter-sum}
  \mathcal S_\beta(F_0,\ldots,F_{K-1})
  :=
  \sum_{z_0,\ldots,z_{K-1}\in\Z^d}
  e\bigl(\mathcal Q_\beta(z)\bigr)
  \prod_{i=0}^{K-1}F_i(z_i).
\end{equation}
Since every
$F_i$ is supported on $[N]^d$, the sum in
\eqref{eq:fixed-parameter-sum} is finite and may equivalently be taken
over $z_0,\ldots,z_{K-1}\in[N]^d$.
For a measurable set $\mathfrak B\subseteq\T^L$, define
\begin{equation}
  \Lambda_N^{\mathfrak B}(F_0,\ldots,F_{K-1})
  :=
  \int_{\mathfrak B}
  \widehat w_R\bigl(\mathcal Q_\beta(v)\bigr)
  \mathcal S_\beta(F_0,\ldots,F_{K-1})\,d\beta.
  \label{eq:arc-restricted-form}
\end{equation}
Taking $\mathfrak B=\T^L$ gives
$\Lambda_N=\Lambda_N^{\T^L}$.

Since $w_R$ is nonnegative and $\sum_n w_R(n)=1$, we have $\norm{\widehat w_R}_{L^\infty(\T)}\leq1$.

\section{Cut operators}

\subsection{The cut operator reduction}

A \emph{cut} $S\mid T$ is a partition of
$\{0,\ldots,K-1\}$ into two nonempty sets. Fix such a cut. For
$z=(z_0,\ldots,z_{K-1})$, write $z_S:=(z_i)_{i\in S}$ and
$z_T:=(z_j)_{j\in T}$.
An edge contained entirely in $S$ or entirely in $T$ contributes to
$\mathcal Q_\beta(z)$ a term depending on only one of these two sets in the cut.
For a crossing edge, with $i\in S$ and $j\in T$, we have
\[
\beta_{ij}\abs{z_i-z_j}^2
=
\beta_{ij}\abs{z_i}^2
-2\beta_{ij}z_i\cdot z_j
+\beta_{ij}\abs{z_j}^2.
\]
Define
\[
\begin{aligned}
  \mathcal Q_\beta^S(z_S)
  &:=
  \sum_{\substack{\{i,i'\}\in\mathcal E\\i,i'\in S}}
  \beta_{ii'}\abs{z_i-z_{i'}}^2
  +
  \sum_{i\in S,\,j\in T}\beta_{ij}\abs{z_i}^2,\\
  \mathcal Q_\beta^{S,T}(z_S,z_T)
  &:=
  -2\sum_{i\in S,\,j\in T}\beta_{ij}z_i\cdot z_j,\\
  \mathcal Q_\beta^T(z_T)
  &:=
  \sum_{\substack{\{j,j'\}\in\mathcal E\\j,j'\in T}}
  \beta_{jj'}\abs{z_j-z_{j'}}^2
  +
  \sum_{i\in S,\,j\in T}\beta_{ij}\abs{z_j}^2.
\end{aligned}
\]
Then
\begin{equation}\label{eq:phase-across-cut}
  \mathcal Q_\beta(z)
  =
  \mathcal Q_\beta^S(z_S)
  +
  \mathcal Q_\beta^{S,T}(z_S,z_T)
  +
  \mathcal Q_\beta^T(z_T).
\end{equation}
The first and third terms in
\eqref{eq:phase-across-cut} depend on variables from only one side of the
cut, while the middle term is the only one involving variables from both
sides.

We view
$e\bigl(\mathcal Q_\beta^{S,T}(z_S,z_T)\bigr)$
as the kernel of an operator taking functions of $z_T$ to functions of
$z_S$. The cut-reduction lemma below shows that the two one-sided phases
$e(\mathcal Q_\beta^S(z_S))$ and $e(\mathcal Q_\beta^T(z_T))$ in
\eqref{eq:phase-across-cut} act by unitary multiplication and therefore
do not affect the relevant $2\to2$ operator norm.

The cut operator reduction will be applied in two different measure spaces.  On the minor arcs,
$(X_i,\mu_i)=\bigl([N]^d,\#\bigr),$
where $\#$ denotes counting measure.  On the major arcs,
$X_i=B_i\times(\Z/q\Z)^d$, where $B_i\subseteq\Z^d$ is finite, and we equip $X_i$ with the measure $\mu_i$ determined by
\[
\int_{X_i}f\,d\mu_i
=
\E_{r_i\bmod q}\sum_{u_i\in B_i}f(u_i,r_i).
\]
Therefore, we formulate the reduction for arbitrary finite measure spaces out of convenience.

For $0\leq i\leq K-1$, let $(X_i,\mu_i)$ be a finite measure space.
Equip $X_S:=\prod_{i\in S}X_i$ with the product measure
$\mu_S:=\bigotimes_{i\in S}\mu_i$. Define $X_T$ and $\mu_T$ in the
same way, with $j\in T$ in place of $i\in S$.
Given a measurable function
$\mathcal Q^{S,T}:X_S\times X_T\longrightarrow\T,$
define
\[
\begin{aligned}
  \mathcal K^{S,T}:L^2(X_T,\mu_T)&\longrightarrow L^2(X_S,\mu_S),\\
  \bigl(\mathcal K^{S,T}h\bigr)(x_S)
  &:=
  \int_{X_T}
  e\!\left(\mathcal Q^{S,T}(x_S,x_T)\right)
  h(x_T)\,d\mu_T(x_T).
\end{aligned}
\]
We call every operator of this form a \emph{cut operator}.

\begin{lemma}[Cut-operator reduction]
\label{lem:cut-cauchy-schwarz}
Let $S\mid T$, $(X_i,\mu_i)$, $\mathcal Q^{S,T}$, and
$\mathcal K^{S,T}$ be as above. Let $\mathcal Q^S:X_S\to\T$
and $\mathcal Q^T:X_T\to\T$ be measurable.
Then, for $f_i\in L^2(X_i,\mu_i)$,
\begin{align*}
  &\abs{
    \int_{X_S}\int_{X_T}
    e\!\left(
    \mathcal Q^S(x_S)
    +\mathcal Q^{S,T}(x_S,x_T)
    +\mathcal Q^T(x_T)
    \right)
    \prod_{i=0}^{K-1}f_i(x_i)
    \,d\mu_T(x_T)\,d\mu_S(x_S)
  }
  \\
  &\hspace{8em}\leq
  \norm{\mathcal K^{S,T}}_{2\to2}
  \prod_{i=0}^{K-1}\norm{f_i}_{L^2(X_i,\mu_i)}.
\end{align*}
\end{lemma}

\begin{proof}
Define multiplication operators by
\[
\bigl(U_Sh\bigr)(x_S)
:=
e\!\left(\mathcal Q^S(x_S)\right)h(x_S),
\qquad
\bigl(U_Th\bigr)(x_T)
:=
e\!\left(\mathcal Q^T(x_T)\right)h(x_T).
\]

For every $h\in L^2(X_T,\mu_T)$,
\begin{align*}
  \bigl(U_S\mathcal K^{S,T}U_Th\bigr)(x_S)
  &=
  e\!\left(\mathcal Q^S(x_S)\right)
  \int_{X_T}
  e\!\left(\mathcal Q^{S,T}(x_S,x_T)\right)
  e\!\left(\mathcal Q^T(x_T)\right)
  h(x_T)\,d\mu_T(x_T)\\
  &=
  \int_{X_T}
  e\!\left(
  \mathcal Q^S(x_S)
  +\mathcal Q^{S,T}(x_S,x_T)
  +\mathcal Q^T(x_T)
  \right)
  h(x_T)\,d\mu_T(x_T).
\end{align*}
Thus $U_S\mathcal K^{S,T}U_T$ is exactly the operator in question.  Since $U_S$ and $U_T$ are unitary,
$\norm{U_S\mathcal K^{S,T}U_T}_{2\to2} = \norm{\mathcal K^{S,T}}_{2\to2}$.
The integral in the statement can now be written as
\[
\int_{X_S}
\left(\prod_{i\in S}f_i(x_i)\right)
\left(
U_S\mathcal K^{S,T}U_T
\left[\prod_{j\in T}f_j\right]
\right)(x_S)\,d\mu_S(x_S).
\]
Applying Cauchy--Schwarz gives the upper bound
\[
\norm{\mathcal K^{S,T}}_{2\to2}
\left\|
\prod_{i\in S}f_i
\right\|_{L^2(X_S,\mu_S)}
\left\|
\prod_{j\in T}f_j
\right\|_{L^2(X_T,\mu_T)}.
\]
The product measures give
\[
\left\|\prod_{i\in S}f_i\right\|_{L^2(X_S,\mu_S)}
=
\prod_{i\in S}\norm{f_i}_{L^2(X_i,\mu_i)},
\qquad
\left\|\prod_{j\in T}f_j\right\|_{L^2(X_T,\mu_T)}
=
\prod_{j\in T}\norm{f_j}_{L^2(X_j,\mu_j)},
\]
which proves the claim.
\end{proof}

\subsection{Tensorization}

We now consider a cut operator which naturally appears in the present problem. Take
$X_i=[R]^d$ with counting measure, and let
$\Theta=(\Theta_{ij})_{i\in S,\,j\in T}\in\T^{\abs S\abs T}$.  For
$z=(z_0,\ldots,z_{K-1})\in([R]^d)^K$, set
$\mathcal Q^{S,T}(z_S,z_T) := \sum_{i\in S,\,j\in T}\Theta_{ij}z_i\cdot z_j$.
The resulting cut operator is
\[
\begin{aligned}
  \mathcal K_{R,d}^{S,T}(\Theta)&:
  \ell^2\bigl(([R]^d)^T\bigr)
  \longrightarrow
  \ell^2\bigl(([R]^d)^S\bigr),\\
  \bigl(\mathcal K_{R,d}^{S,T}(\Theta)h\bigr)(z_S)
  &:={}
  \sum_{z_T\in([R]^d)^T}
  e\!\left(
  \sum_{i\in S,\,j\in T}\Theta_{ij}z_i\cdot z_j
  \right)h(z_T).
\end{aligned}
\]
We identify this $d$-dimensional cut operator with an
$R^{d\abs S}\times R^{d\abs T}$ matrix:
\[
\mathcal K_{R,d}^{S,T}(\Theta)
=
\left(
e\!\left(
\sum_{i\in S,\,j\in T}\Theta_{ij}z_i\cdot z_j
\right)
\right)_{\substack{
    z_S\in([R]^d)^S\\
    z_T\in([R]^d)^T
}}.
\]

For the one-dimensional case, we write
$\mathcal K_R^{S,T}(\Theta) := \mathcal K_{R,1}^{S,T}(\Theta)$
and define the normalized cut operator norm by
$m_R^{S,T}(\Theta) := R^{-K/2}\norm{\mathcal K_R^{S,T}(\Theta)}_{2\to2}$.
By Cauchy--Schwarz, the operator norm of a finite matrix is at most its
Hilbert--Schmidt norm.  Since the $R^K$ entries of
$\mathcal K_R^{S,T}(\Theta)$ all have modulus one,
\[
\norm{\mathcal K_R^{S,T}(\Theta)}_{2\to2}
\leq
\norm{\mathcal K_R^{S,T}(\Theta)}_{\mathrm{HS}}
=R^{K/2},
\]
and therefore $0\leq m_R^{S,T}(\Theta)\leq1$.

The following lemma gives a relationship between the operator norms of $\mathcal K_R^{S,T}$ and $\mathcal K_{R,d}^{S,T}$. This reduction plays a key role throughout the proof.

\begin{lemma}\label{lem:cut-tensorization}
For every cut $S\mid T$,
\[
\bigl\|\mathcal K_{R,d}^{S,T}(\Theta)\bigr\|_{2\to2}
=
\bigl\|\mathcal K_R^{S,T}(\Theta)\bigr\|_{2\to2}^{d}
=
R^{Kd/2}m_R^{S,T}(\Theta)^d.
\]
\end{lemma}

\begin{proof}
Write
$z_i=(z_{i,1},\ldots,z_{i,d})$. For every $1\leq\ell\leq d$, put
$z_{S,\ell}:=(z_{i,\ell})_{i\in S}\in[R]^S$ and $z_{T,\ell}:=(z_{j,\ell})_{j\in T}\in[R]^T$.
Thus the row index $z_S$ is
$(z_{S,1},\ldots,z_{S,d})$, and the column index $z_T$ is the same
data as $(z_{T,1},\ldots,z_{T,d})$.  In this way, the
$(z_S,z_T)$-entry of $\mathcal K_{R,d}^{S,T}(\Theta)$ is
\[
\begin{aligned}
  e\!\left(
  \sum_{i\in S,\,j\in T}\Theta_{ij}z_i\cdot z_j
  \right)=
  e\!\left(
  \sum_{\ell=1}^d
  \sum_{i\in S,\,j\in T}
  \Theta_{ij}z_{i,\ell}z_{j,\ell}
  \right)=
  \prod_{\ell=1}^d
  e\!\left(
  \sum_{i\in S,\,j\in T}
  \Theta_{ij}z_{i,\ell}z_{j,\ell}
  \right).
\end{aligned}
\]
For each $\ell$, the corresponding factor is the
$(z_{S,\ell},z_{T,\ell})$-entry of
$\mathcal K_R^{S,T}(\Theta)$.  This is precisely the definition of a
$d$-fold tensor product:
\[
\mathcal K_{R,d}^{S,T}(\Theta)
=
\bigl(\mathcal K_R^{S,T}(\Theta)\bigr)^{\otimes d}.
\]
The standard multiplicativity of the $2\to2$ operator norm under tensor
products now gives
\[
\bigl\|\mathcal K_{R,d}^{S,T}(\Theta)\bigr\|_{2\to2}
=
\bigl\|\mathcal K_R^{S,T}(\Theta)\bigr\|_{2\to2}^{d}.
\]
Finally, by the definition of $m_R^{S,T}(\Theta)$,
\[
\bigl\|\mathcal K_R^{S,T}(\Theta)\bigr\|_{2\to2}^{d}
=
R^{Kd/2}m_R^{S,T}(\Theta)^d.
\]
\end{proof}

The cut operators $\mathcal K^{S,T}_R$ appear in both the major and minor arc estimates.  On the minor arcs, Lemma \ref{lem:cut-cauchy-schwarz} and 
Lemma~\ref{lem:cut-tensorization} reduce the analysis to understanding 
$m_N^{S,T}(-2\beta_{S,T})$.  On the major arcs, the same reductions lead us to 
$m_{2N}^{S,T}(-2\theta_{S,T})$ paired with the complete-sum factor
$\abs{\operatorname{im}(2a_{S,T})}^{-d/2}$. Here we used the following useful notation: for $b=(b_e)_{e\in\mathcal E}$, put
$b_{ij}=b_{ji}:=b_{\{i,j\}}$ and write
$b_{S,T}:=(b_{ij})_{i\in S,\,j\in T}$ for its matrix of frequencies corresponding to the crossing
edges.

\subsection{Reduction to one crossing edge}

For
$\theta\in\T$, define 
$\mathcal K_R(\theta) := \bigl(e(\theta rs)\bigr)_{r,s\in[R]}$.
Write $m_R(\theta):=R^{-1}\norm{\mathcal K_R(\theta)}_{2\to2}$ for
its normalized norm. The matrix $\mathcal K_R(\theta)$ is the one-dimensional cut operator associated with a cut having one vertex on each
side.  Notice $\mathcal K_R(\theta)$ is an $R\times R$ matrix acting on
$\ell^2([R])$, and its $r$th output coordinate is
$\sum_{s\in[R]}e(\theta rs)h(s)$.  Its Hilbert--Schmidt norm is $R$, so
$0\leq m_R(\theta)\leq1$. In addition, we have 
$m_R(-\theta)=m_R(\theta)$.

We require Schur's test. If
$M=(M_{ab})$ is a $p\times q$ matrix whose entries are operators, then
$\norm{M}_{2\to2}\leq\sqrt{pq}\max_{a,b}\norm{M_{ab}}_{2\to2}$. 

\begin{lemma}
\label{lem:single-edge-domination}
Let $\Theta\in\T^{\abs S\abs T}$ and let $i\in S$, $j\in T$.  Then
$m_R^{S,T}(\Theta)\leq m_R(\Theta_{ij}).$
\end{lemma}

\begin{proof}
Abbreviate $\mathcal K:=\mathcal K_R^{S,T}(\Theta)$.
Write $\widehat z_S:=(z_k)_{k\in S\setminus\{i\}}$ and
$\widehat z_T:=(z_\ell)_{\ell\in T\setminus\{j\}}$.
For fixed $\widehat z_S$ and $\widehat z_T$, define
\[
\mathcal K[\widehat z_S,\widehat z_T]
:=
\left(
e\!\left(\mathcal Q^{S,T}(z_S,z_T)\right)
\right)_{z_i,z_j\in[R]},
\]
where the coordinates contained in $\widehat z_S$ and
$\widehat z_T$ are fixed.  This is the $R\times R$ block of
$\mathcal K$ in which $z_i$ indexes the rows and $z_j$ indexes the
columns.

Every row is given by $(\widehat z_S,z_i)$ and every column by
$(\widehat z_T,z_j)$.  Grouping them
gives the block matrix
\[
\mathcal K
=
\left(
\mathcal K[\widehat z_S,\widehat z_T]
\right)_{\substack{
    \widehat z_S\in[R]^{S\setminus\{i\}}\ \\
    \widehat z_T\in[R]^{T\setminus\{j\}}\ }}.
\]
Thus the block matrix has $R^{\abs S-1}$ block rows and 
$R^{\abs T-1}$ block columns, and every block is an $R\times R$ matrix.

The definition of
$\mathcal Q^{S,T}$ gives

\begin{equation}\label{phaseexpansion}
  \mathcal Q^{S,T}(z_S,z_T)
  =\Theta_{ij}z_iz_j
  +z_i\sum_{\ell\in T\setminus\{j\}}\Theta_{i\ell}z_\ell
  +z_j\sum_{k\in S\setminus\{i\}}\Theta_{kj}z_k+\sum_{\substack{k\in S\setminus\{i\}\\
      \ell\in T\setminus\{j\}}}
  \Theta_{k\ell}z_kz_\ell.
\end{equation}

For $t\in\T$, set
\[
D_R(t):=
\begin{pmatrix}
  e(t)&&0\\
  &\ddots&\\
  0&&e(Rt)
\end{pmatrix}.
\]
Consequently the block has the factorization
\[
\begin{aligned}
  \mathcal K[\widehat z_S,\widehat z_T]
  ={}&
  e\!\left(
  \sum_{\substack{k\in S\setminus\{i\}\\
      \ell\in T\setminus\{j\}}}
  \Theta_{k\ell}z_kz_\ell
  \right)
  D_R\!\left(
  \sum_{\ell\in T\setminus\{j\}}\Theta_{i\ell}z_\ell
  \right)
  \mathcal K_R(\Theta_{ij})
  D_R\!\left(
  \sum_{k\in S\setminus\{i\}}\Theta_{kj}z_k
  \right).
\end{aligned}
\]
The scalar corresponds to the last term in \eqref{phaseexpansion}.  The diagonal matrices give the terms linear in $z_i$ and $z_j$, respectively, while
$\mathcal K_R(\Theta_{ij})$ gives the term $\Theta_{ij}z_iz_j$. The scalar has modulus one and the
two diagonal matrices are unitary, so  
\[
\norm{\mathcal K[\widehat z_S,\widehat z_T]}_{2\to2}
=
\norm{\mathcal K_R(\Theta_{ij})}_{2\to2}.
\]

Applying Schur's test yields
\[
\norm{\mathcal K}_{2\to2}
\leq
R^{(\abs S-1)/2}R^{(\abs T-1)/2}
\norm{\mathcal K_R(\Theta_{ij})}_{2\to2}
=R^{(K-2)/2}\norm{\mathcal K_R(\Theta_{ij})}_{2\to2}.
\]
Thus, 
\[
m_R^{S,T}(\Theta)
=R^{-K/2}\norm{\mathcal K}_{2\to2}
\leq R^{-1}\norm{\mathcal K_R(\Theta_{ij})}_{2\to2}
=m_R(\Theta_{ij}).
\]
\end{proof}

\subsection{Proof of the mean value estimate}

We now prove the mean value estimate stated in
Lemma~\ref{lem:rectangular-moment}.  Its arithmetic input is the matrix
count from Appendix~\ref{app:matrix-count}, which is derived from Kim's
mean-value formula for primitive lattice flags.

\begin{proof}[Proof of Lemma~\ref{lem:rectangular-moment}]
Abbreviate
$\mathcal K:=\mathcal K_R^{S,T}(\Theta).$
The matrix $\mathcal K\mathcal K^*$ is Hermitian and positive semidefinite.
If its eigenvalues are
$\lambda_1,\ldots,\lambda_{R^{\abs S}}\geq0$, then
\begin{equation}\label{tracereduction}
  \norm{\mathcal K}_{2\to2}^{2K+4}
  =
  \left(\max_a\lambda_a\right)^{K+2}
  \leq
  \sum_a\lambda_a^{K+2}
  =
  \operatorname{tr}\bigl((\mathcal K\mathcal K^*)^{K+2}\bigr).
\end{equation}
In this way we reduce the proof to analyzing
$\operatorname{tr}\bigl((\mathcal K\mathcal K^*)^{K+2}\bigr)$.

For $z_S,\widetilde z_S\in[R]^S$, ordinary matrix multiplication
gives the $(z_S,\widetilde z_S)$-entry of
$\mathcal K\mathcal K^*$ as
\begin{align}
  (\mathcal K\mathcal K^*)_{z_S,\widetilde z_S}
  &=
  \sum_{z_T\in[R]^T}
  e\!\left(
  \sum_{i\in S,\,j\in T}\Theta_{ij}z_i z_j
  \right)
  \overline{
    e\!\left(
    \sum_{i\in S,\,j\in T}\Theta_{ij}\widetilde z_i z_j
    \right)}\notag\\
  &=
  \sum_{z_T\in[R]^T}
  e\!\left(
  \sum_{i\in S,\,j\in T}
  \Theta_{ij}(z_i-\widetilde z_i)z_j
  \right),\label{KKstarexpansion}
\end{align}
and expanding the trace gives
\begin{equation}\label{traceexpansion}
  \operatorname{tr}\bigl((\mathcal K\mathcal K^*)^{K+2}\bigr)
  =
  \sum_{z_S^{(1)},\ldots,z_S^{(K+2)}\in[R]^S}
  (\mathcal K\mathcal K^*)_{z_S^{(1)},z_S^{(2)}}\cdots
  (\mathcal K\mathcal K^*)_{z_S^{(K+2)},z_S^{(1)}}.
\end{equation}

Combining \eqref{KKstarexpansion} with \eqref{traceexpansion} yields
\begin{align*}
  &\operatorname{tr}\bigl((\mathcal K\mathcal K^*)^{K+2}\bigr)\\
  &=
  \sum_{\substack{z_S^{(1)},\ldots,z_S^{(K+2)}\\\in[R]^S}}
  \sum_{\substack{z_T^{(1)},\ldots,z_T^{(K+2)}\\\in[R]^T}}
  e\!\left(
  \sum_{i\in S,\,j\in T}\Theta_{ij}
  \left[
  \sum_{h=1}^{K+1}
  \bigl(z_i^{(h)}-z_i^{(h+1)}\bigr)z_j^{(h)}
  +
  \bigl(z_i^{(K+2)}-z_i^{(1)}\bigr)z_j^{(K+2)}
  \right]
  \right).
\end{align*}

Now integrate each $\Theta_{ij}$ over $\T$ and apply the orthogonality relation. The equations imposed by orthogonality are precisely the entries of the
matrix equation
\begin{equation}\label{matrixeq}
  \sum_{h=1}^{K+1}
  \bigl(z_S^{(h)}-z_S^{(h+1)}\bigr)
  \bigl(z_T^{(h)}\bigr)^{\mathsf T}
  +
  \bigl(z_S^{(K+2)}-z_S^{(1)}\bigr)
  \bigl(z_T^{(K+2)}\bigr)^{\mathsf T}
  =
  0.
\end{equation}
Consequently, the integral of the trace counts the solutions to this
matrix equation with $z_S^{(h)}\in[R]^S$ and $z_T^{(h)}\in[R]^T$
for every $1\leq h\leq K+2$.

We next put this constraint into the form required by
Lemma~\ref{lem:matrix-orthogonality}. The differences on the $S$-side
telescope:
\[
\bigl(z_S^{(1)}-z_S^{(2)}\bigr)
+\cdots+
\bigl(z_S^{(K+2)}-z_S^{(1)}\bigr)
=
0.
\]
We subtract $z_T^{(K+2)}$ from every $T$-variable
in \eqref{matrixeq} so that the equation becomes
\[
\sum_{h=1}^{K+1}
\bigl(z_S^{(h)}-z_S^{(h+1)}\bigr)
\bigl(z_T^{(h)}-z_T^{(K+2)}\bigr)^{\mathsf T}
=
0.
\]

For $1\leq h\leq K+1$, set
$U_h:=z_S^{(h)}-z_S^{(h+1)}\in\Z^S$ and $V_h:=z_T^{(h)}-z_T^{(K+2)}\in\Z^T$.
Let $U\in\Z^{\abs S\times(K+1)}$ and $V\in\Z^{\abs T\times(K+1)}$
be the matrices whose $h$th columns are $U_h$ and $V_h$, respectively.
Since
$UV^{\mathsf T} = \sum_{h=1}^{K+1}U_hV_h^{\mathsf T},$
equation \eqref{matrixeq} is precisely
$UV^{\mathsf T}=0.$

Once $U$ and $V$ are fixed, the variables $z_S^{(h)}$ and
$z_T^{(h)}$ for $1\leq h\leq K+2$ are determined by the two points
$z_S^{(1)}\in[R]^S$ and $z_T^{(K+2)}\in[R]^T$. Indeed, for every $1\leq h\leq K+1$,
$z_S^{(h+1)}=z_S^{(h)}-U_h$ and $z_T^{(h)}=z_T^{(K+2)}+V_h$.
Thus each pair $U,V$ arises from at most
$R^{\abs S}R^{\abs T}=R^K$
choices of these variables. Some choices of the two points may
yield variables outside $[R]^S$ or $[R]^T$, which only
decreases this count.

All entries of $U$ and $V$ lie in $[-R,R]$. Moreover, the fact that the cut is balanced gives
$2\max(\abs S,\abs T)\leq K+1,$
while
$\abs S+\abs T=K<K+1.$
Thus Lemma~\ref{lem:matrix-orthogonality} applies with $n=K+1$ and
shows that the number of possible pairs $U,V$ is
\[
\ll_K
R^{(K+1)(\abs S+\abs T)-2\abs S\abs T}
=
R^{K(K+1)-2\abs S\abs T}.
\]
Combining this with the $R^K$ choices of the two base points gives
\begin{equation}\label{tracebound}
  \int_{\T^{\abs S\abs T}}
  \operatorname{tr}\bigl((\mathcal K\mathcal K^*)^{K+2}\bigr)\,d\Theta
  \ll_K
  R^{K(K+2)-2\abs S\abs T}.
\end{equation}

Finally, \eqref{tracereduction} yields 
\[
m_R^{S,T}(\Theta)^{2K+4}
=
R^{-K(K+2)}
\norm{\mathcal K}_{2\to2}^{2K+4}
\leq
R^{-K(K+2)}
\operatorname{tr}\bigl((\mathcal K\mathcal K^*)^{K+2}\bigr).
\]
Integrating this inequality and applying \eqref{tracebound}
gives
\[
\int_{\T^{\abs S\abs T}}
m_R^{S,T}(\Theta)^{2K+4}\,d\Theta
\ll_K
R^{-2\abs S\abs T},
\]
which is \eqref{eq:rectangular-moment}.
\end{proof}

\subsection{Balanced cuts and fractional H\"older}

Recall that a cut $S\mid T$ of $\{0,\ldots,K-1\}$ is
\emph{balanced} if the sizes of $S$ and $T$ differ by at most one.
Let $\mathscr B_K$ denote the set of all balanced cuts, with
$S\mid T$ and $T\mid S$ regarded as the same cut.  Give every cut in
$\mathscr B_K$ weight
$w_K := L/(\lfloor K^2/4\rfloor\,\abs{\mathscr B_K}).$ For a cut $S\mid T$, regard the edges crossing $S\mid T$ as a subset of
$\mathcal E$.  A \emph{fractional cover} of $\mathcal E$ by such sets
assigns a nonnegative weight to each cut so that the total weight of the
cuts crossed by any fixed edge is at least one.  The following lemma
shows that this is satisfied for the weights just described. 

\begin{lemma}\label{lem:balanced-cover}
For every edge $e\in\mathcal E$,
\begin{equation}\label{eq:fractional-cover}
  w_K\#\{S\mid T\in\mathscr B_K:e\text{ crosses }S\mid T\}
  =1.
\end{equation}
\end{lemma}

\begin{proof}
By symmetry, every edge crosses the same number
of balanced cuts.

Count pairs consisting of a balanced cut and one of its crossing
edges.  Choosing the cut first gives
$\abs{\mathscr B_K}\lfloor K^2/4\rfloor$ pairs, since every balanced
cut has $\lfloor K^2/4\rfloor$ crossing edges.  Choosing the edge first
gives
\[
L\#\{S\mid T\in\mathscr B_K:e\text{ crosses }S\mid T\}
\]
pairs, since there are $L$ edges and the number of balanced cuts
crossed by an edge is independent of the edge.  Therefore
\[
L\#\{S\mid T\in\mathscr B_K:e\text{ crosses }S\mid T\}
=
\abs{\mathscr B_K}\lfloor K^2/4\rfloor.
\]
Multiplying by $w_K/L$ proves \eqref{eq:fractional-cover}.
\end{proof}

\noindent\begin{minipage}{\textwidth}
For $K=4$, the cover can be read directly from the three diagrams
below.  The orange and blue vertices lie on the two sides of each cut.
Every edge is internal in exactly one panel and crosses the cut in the
other two.  Since $w_4=1/2$, every edge therefore has total crossing
weight one.

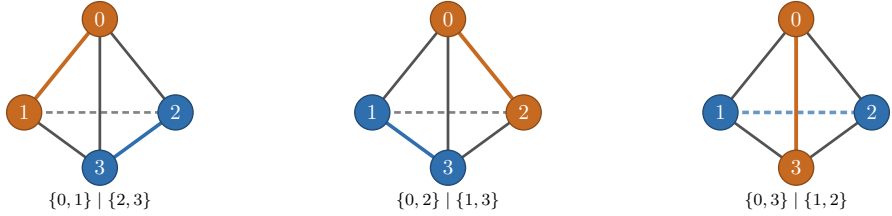
\begin{center}
\begin{tikzpicture}
  \begin{scope}[shift={(-4.6,0)},scale=0.80,transform shape]
    \coordinate (a0) at (0,1.85);
    \coordinate (a1) at (-1.25,0.305);
    \coordinate (a2) at (1.25,0.305);
    \coordinate (a3) at (0,-0.61);
    \draw[crossinghidden] (a1) -- (a2);
    \draw[sinternal] (a0) -- (a1);
    \draw[crossing] (a0) -- (a2);
    \draw[crossing] (a0) -- (a3);
    \draw[crossing] (a1) -- (a3);
    \draw[tinternal] (a2) -- (a3);
    \node[svertex] at (a0) {$0$};
    \node[svertex] at (a1) {$1$};
    \node[tvertex] at (a2) {$2$};
    \node[tvertex] at (a3) {$3$};
    \node[font=\scriptsize] at (0,-1.16)
    {$\{0,1\}\mid\{2,3\}$};
  \end{scope}

  \begin{scope}[scale=0.80,transform shape]
    \coordinate (b0) at (0,1.85);
    \coordinate (b1) at (-1.25,0.305);
    \coordinate (b2) at (1.25,0.305);
    \coordinate (b3) at (0,-0.61);
    \draw[crossinghidden] (b1) -- (b2);
    \draw[crossing] (b0) -- (b1);
    \draw[sinternal] (b0) -- (b2);
    \draw[crossing] (b0) -- (b3);
    \draw[tinternal] (b1) -- (b3);
    \draw[crossing] (b2) -- (b3);
    \node[svertex] at (b0) {$0$};
    \node[tvertex] at (b1) {$1$};
    \node[svertex] at (b2) {$2$};
    \node[tvertex] at (b3) {$3$};
    \node[font=\scriptsize] at (0,-1.16)
    {$\{0,2\}\mid\{1,3\}$};
  \end{scope}

  \begin{scope}[shift={(4.6,0)},scale=0.80,transform shape]
    \coordinate (c0) at (0,1.85);
    \coordinate (c1) at (-1.25,0.305);
    \coordinate (c2) at (1.25,0.305);
    \coordinate (c3) at (0,-0.61);
    \draw[draw=crossblue,line width=1.35pt,
    dash pattern=on 2.5pt off 1.6pt,opacity=0.72] (c1) -- (c2);
    \draw[crossing] (c0) -- (c1);
    \draw[crossing] (c0) -- (c2);
    \draw[sinternal] (c0) -- (c3);
    \draw[crossing] (c1) -- (c3);
    \draw[crossing] (c2) -- (c3);
    \node[svertex] at (c0) {$0$};
    \node[tvertex] at (c1) {$1$};
    \node[tvertex] at (c2) {$2$};
    \node[svertex] at (c3) {$3$};
    \node[font=\scriptsize] at (0,-1.16)
    {$\{0,3\}\mid\{1,2\}$};
  \end{scope}
\end{tikzpicture}
\end{center}
\end{minipage}
\par\medskip
The fractional cover identity
\eqref{eq:fractional-cover} is exactly the hypothesis in the following
special case of Finner's fractional form of H\"older's inequality
\cite{Finner}. 

\begin{lemma}[Fractional H\"older inequality]
\label{lem:fractional-holder}
Let $\mathcal E$ be a finite set.  For each $e\in\mathcal E$, let
$(X_e,\mu_e)$ be a $\sigma$-finite measure space, and let
$\mu:=\bigotimes_{e\in\mathcal E}\mu_e$
be the product measure.  Let $J$ be a positive integer.  For each
$1\leq j\leq J$, choose a nonempty subset
$\mathcal E_j\subseteq\mathcal E$, a weight $c_j>0$, and a nonnegative
integrable function
$f_j:\prod_{e\in\mathcal E_j}X_e\longrightarrow[0,\infty).$
Denote the product measure on the domain of $f_j$ by
$\mu_j:=\bigotimes_{e\in\mathcal E_j}\mu_e.$
Suppose that for every $e\in \mathcal E$, we have 
\[
\sum_{\substack{1\leq j\leq J\\e\in\mathcal E_j}}c_j=1.
\]
Then, writing $x=(x_e)_{e\in\mathcal E}$,
\[
\int
\prod_{j=1}^J
f_j\bigl((x_e)_{e\in\mathcal E_j}\bigr)^{c_j}\,d\mu(x)
\leq
\prod_{j=1}^J
\left(\int f_j\,d\mu_j\right)^{c_j}.
\]
\end{lemma}
This is \cite[Theorem~2.1]{Finner}, applied with exponents
$1/c_j$ to the functions $f_j^{c_j}$.

\section{The minor arcs}

For $2\leq Q\leq R$, define the major arcs by
\[
 \mathfrak M_Q
 :=\bigcup_{1\leq q\leq Q}
 \ \bigcup_{\substack{a\in(\Z/q\Z)^L\\(a,q)=1}}
 \left\{\beta\in\T^L:\left\|\beta-\frac aq\right\|_\infty
 \leq\frac{Q}{qR^2}\right\}.
\]
Equivalently, $\beta\in\mathfrak M_Q$ precisely when $\|q\beta\|_\infty\leq Q/R^2$
for some integer $1\leq q\leq Q$. The minor arcs are
$\mathfrak m_Q:=\T^L\setminus\mathfrak M_Q$.

The goal of the section is to prove the following minor arc estimate.
\begin{proposition}[minor arc estimate]\label{prop:minor arcs}
Assume $2\leq Q\leq R$ and $d\geq4K+4$.  There is
$\gamma>0$ depending on $K$ and $d$ such that, for all 
$F_0,\ldots,F_{K-1}:[N]^d\to\mathbb C$, we have 
\begin{align}
  \abs{\Lambda_N^{\mathfrak m_Q}(F_0,\ldots,F_{K-1})}
  &\ll_{\sigma,\Delta,d}
  Q^{-\gamma/L}R^{Kd/2-2L}
  \prod_{i=0}^{K-1}\norm{F_i}_{\ell^2(\Z^d)}.
  \label{eq:global-minor}
\end{align}
\end{proposition}

\subsection{Applying the cut operator reduction}

For a cut
$S\mid T$, take $X_i=[N]^d$ with counting measure and set $f_i=F_i$.
The part of the phase corresponding to the crossing edges is
\[
\mathcal Q_\beta^{S,T}(z_S,z_T)
=
-2\sum_{i\in S,\,j\in T}\beta_{ij}z_i\cdot z_j,
\]
so the cut operator in Lemma~\ref{lem:cut-cauchy-schwarz} is exactly
$\mathcal K_{N,d}^{S,T}(-2\beta_{S,T})$.  Thus, 
\[
|\mathcal S_\beta(F_0,\ldots,F_{K-1})|
\leq
\|\mathcal K_{N,d}^{S,T}(-2\beta_{S,T})\|_{2\to2}
\prod_{i=0}^{K-1}\|F_i\|_{\ell^2(\Z^d)}.
\]

This yields the following case of the cut-operator reduction discussed above.
\begin{proposition}
\label{prop:cut-operator-estimate}
Let $F_0,\ldots,F_{K-1}:[N]^d\to\mathbb C$.
Let $\beta\in\T^L$ and $S\mid T$ be a cut of the vertex set of $\Delta$. Then, we have 
\[
  |\mathcal S_\beta(F_0,\ldots,F_{K-1})|
  \leq
  N^{Kd/2}
  m_N^{S,T}\left(-2\beta_{S,T}\right)^d
  \prod_{i=0}^{K-1}\|F_i\|_{\ell^2(\Z^d)}.
\]
\end{proposition}
\subsection{The large-sieve estimate}

We now prove Lemma~\ref{lem:scalar-minor}, which is a simple application of the analytic large sieve.

\begin{proof}[Proof of Lemma~\ref{lem:scalar-minor}]
Fix $\beta\in\mathfrak m_Q$ and put $H=(Q/2)^{1/L}$. For each edge $e$,
Dirichlet's theorem gives coprime integers $a_e,r_e$ with
$1\leq r_e\leq P^2/H$ and $\|-2\beta_e-a_e/r_e\|_\T\leq H/(r_eP^2)$.
If every $r_e\leq H$, then $q=2\prod_e r_e\leq Q$ and
$\|q\beta_e\|_\T\leq(\prod_{f\ne e}r_f)\|2r_e\beta_e\|_\T\leq H^L/P^2$
for every $e$. Since $H^L/P^2\leq Q/R^2$, this contradicts $\beta\in\mathfrak m_Q$.

Choose an edge with $r_e>H$, and write $\xi=-2\beta_e$ and $a/q=a_e/r_e$.
Then $H<q\leq P^2/H$ and $\|\xi-a/q\|_\T\leq H/(qP^2)\leq1/q^2$.
If $q<8$, then $H<8$, and the conclusion follows from
$m_P(\xi)^2\leq1\leq8H^{-1}\ll Q^{-1/L}$.
We may therefore suppose that $q\geq8$.

Partition $[P]$ into $j\ll 1+P/q$ intervals
$I_1,\ldots,I_j$, each of diameter at most $q/8$. Thus
$|u-u'|\leq\frac{q}{8}$
for every $1\leq r\leq j$ and every pair $u,u'\in I_r$.
Fix one such interval $I_r$. Suppose $u,u'\in I_r$ are distinct.
Since $(a,q)=1$ and $|u-u'|\leq q/8<q$, the integer
$a(u-u')$ is not divisible by $q$, so
$\left\|a(u-u')/q\right\|_{\T}\geq 1/q.$
Moreover,
\[
\abs{u-u'}\norm{\xi-a/q}_{\T}
\leq
\frac q8\cdot\frac1{q^2}
=
\frac1{8q}.
\]
It follows from the reverse triangle inequality that
\[
\begin{aligned}
  \|\xi u-\xi u'\|_{\T}
  &=\|\xi(u-u')\|_{\T}\\
  &\geq
  \left\|\frac{a(u-u')}{q}\right\|_{\T}
  -
  \left\|(u-u')\left(\xi-\frac aq\right)\right\|_{\T}\\
  &\geq
  \left\|\frac{a(u-u')}{q}\right\|_{\T}
  -
  |u-u'|\left\|\xi-\frac aq\right\|_{\T}\\
  &\geq \frac1q-\frac1{8q}
  =\frac7{8q}.
\end{aligned}
\]
Thus the points $\{\xi u:u\in I_r\}$ are pairwise separated in
$\T$ by at least $7/(8q)$.

Applying the analytic large sieve
\cite[Theorem~1, equation~(1.4)]{MontgomeryVaughan} to the
trigonometric polynomial
$\sum_{s\in[P]}h(s)e(sx)$
at the points $\{\xi u:u\in I_r\}$ and using their separation
yields
\[
\sum_{u\in I_r}
\bigg|
\sum_{s\in[P]}e(\xi us)h(s)
\bigg|^2
\leq
\left(P+\frac{8q}{7}\right)
\|h\|_{\ell^2([P])}^2.
\]
Summing this inequality over $r\in[j]$ and using the fact that the
intervals $I_r$ partition $[P]$, we obtain
\[
\sum_{u\in[P]}
\bigg|
\sum_{s\in[P]}e(\xi us)h(s)
\bigg|^2
\ll
\left(1+\frac Pq\right)(P+q)\|h\|_{\ell^2([P])}^2.
\]
Taking the supremum over $h\in\ell^2([P])$ with
$\|h\|_{\ell^2([P])}\leq1$ yields
\[
\|\mathcal K_P(\xi)\|_{2\to2}^2
\ll
\left(1+\frac Pq\right)(P+q).
\]
Since
$m_P(\xi) = P^{-1}\norm{\mathcal K_P(\xi)}_{2\to2},$
we conclude that
\[
m_P(\xi)^2
\ll
P^{-2}\left(1+\frac Pq\right)(P+q)
\ll
q^{-1}+P^{-1}+qP^{-2}
\ll H^{-1}\ll Q^{-1/L}.
\]
Here we used $H<q\leq P^2/H$ and $H\leq P$.
\end{proof}

\subsection{Proof of the minor arc estimate}

\begin{proof}[Proof of Proposition~\ref{prop:minor arcs}]
Put
$\tau := d\lfloor K^2/4\rfloor/(2L).$
Since $2\tau w_K=d/\abs{\mathscr B_K}$, the geometric mean of
the cut factors in Proposition~\ref{prop:cut-operator-estimate}
is $I_N(\beta)$ as defined in \eqref{eq:balanced-cut-mean}. Thus
\begin{align}
  \abs{\mathcal S_\beta(F_0,\ldots,F_{K-1})}
  \leq{}&
  N^{Kd/2}
  I_N(\beta)
  \prod_{i=0}^{K-1}\norm{F_i}_{\ell^2(\Z^d)}.
  \label{eq:geometric-cut-bound}
\end{align}
The dimension threshold $d\geq 4K+4$ yields 
\[
\tau
\geq
\frac{(4K+4)(K^2-1)}{8L}
=
\frac{(K+1)^2}{K}
>
K+2.
\]
Set $\gamma:=\min\{1,\tau-K-2\}>0$.
From
\eqref{eq:arc-restricted-form} and \eqref{eq:geometric-cut-bound}, we see that it suffices to show
\begin{equation}\label{eq:minor-cut-integral}
  \int_{\mathfrak m_Q}
  I_N(\beta)\,d\beta
  \ll_{K,d}
  Q^{-\gamma/L}N^{-2L}.
\end{equation}

For each integer $P\geq1$, $\beta\in\T^L$, and  $e\in\mathcal E$,
Lemma~\ref{lem:single-edge-domination} gives
$m_P^{S,T}(-2\beta_{S,T})\leq m_P(-2\beta_e)$ whenever $e$ is a crossing edge for $S\mid T$.
Since $\tau-\gamma\geq K+2$, we obtain
\begin{equation}\label{eq:edge-extraction}
 I_P(\beta)\leq m_P(-2\beta_e)^{2\gamma}
 \prod_{S\mid T\in\mathscr B_K}
 m_P^{S,T}(-2\beta_{S,T})^{(2K+4)w_K},
\end{equation}
where we also used the fact that the cut operator norms are normalized. 

Now we apply Lemma~\ref{lem:fractional-holder}. For the application, take $X_e=\T$ with normalized Haar measure. Thus $x_e=\beta_e$ and
$d\mu(\beta)=d\beta$ on $\T^L$. Put $J=\abs{\mathscr B_K}$ and assign an 
index in  $\{1,\ldots,J\}$ to each balanced cut. For the index $j$ corresponding
to $S\mid T$, let $\mathcal E_j$ be the set of edges crossing $S\mid T$
and take $c_j=w_K$. The tuple $(\beta_e)_{e\in\mathcal E_j}$ is
$\beta_{S,T}$, and $d\mu_j=d\beta_{S,T}$ is normalized Haar measure
on $\T^{\abs S\abs T}$. Define
\[
 f_j(\beta_{S,T}):=m_P^{S,T}(-2\beta_{S,T})^{2K+4}.
\]
For every $e\in\mathcal E$,
\eqref{eq:fractional-cover} gives
\[
 \sum_{j:\,e\in\mathcal E_j}c_j
 =\sum_{\substack{S\mid T\in\mathscr B_K\\e\text{ crosses }S\mid T}}w_K=1,
\]
which is the hypothesis of Lemma~\ref{lem:fractional-holder}. By substituting the
functions $f_j$ and weights $c_j$, we have
\begin{align}
 \int_{\T^L}\prod_{S\mid T\in\mathscr B_K}
 m_P^{S,T}(-2\beta_{S,T})^{(2K+4)w_K}\,d\beta
 &\leq \prod_{S\mid T\in\mathscr B_K}
 \left(\int_{\T^{\abs S\abs T}}m_P^{S,T}(-2\beta_{S,T})^{2K+4}\,d\beta_{S,T}\right)^{w_K}\notag\\
 &=\prod_{S\mid T\in\mathscr B_K}
 \left(\int_{\T^{\abs S\abs T}}m_P^{S,T}(\Theta)^{2K+4}\,d\Theta\right)^{w_K}\notag\\
 &\ll_K\prod_{S\mid T\in\mathscr B_K}P^{-2\abs S\abs T w_K}
 =P^{-2L}.
 \label{eq:full-cut-moment}
\end{align}
The last line follows from
Lemma~\ref{lem:rectangular-moment} and the identity
$\sum_{S\mid T\in\mathscr B_K}\abs S\abs T w_K=L$, obtained by summing
\eqref{eq:fractional-cover} over all $L$ edges.

For each $\beta\in\mathfrak m_Q$, use Lemma~\ref{lem:scalar-minor} with $P=N$
to choose $e$ in \eqref{eq:edge-extraction}. Integrating and applying
\eqref{eq:full-cut-moment}, we obtain
\[
 \int_{\mathfrak m_Q}I_N(\beta)\,d\beta
 \ll_{K,d}Q^{-\gamma/L}N^{-2L},
\]
as required in \eqref{eq:minor-cut-integral}. Combining this estimate
with \eqref{eq:arc-restricted-form}, $\norm{\widehat w_R}_{L^\infty(\T)}\leq1$,
and \eqref{eq:geometric-cut-bound} gives
\[
 \abs{\Lambda_N^{\mathfrak m_Q}(F_0,\ldots,F_{K-1})}
 \ll_{\Delta,d} Q^{-\gamma/L}N^{Kd/2-2L}
 \prod_{i=0}^{K-1}\norm{F_i}_{\ell^2(\Z^d)}.
\]
Finally, $N\geq2/\sigma$ implies
$R=\lfloor\sigma N\rfloor\geq\sigma N/2$, which gives \eqref{eq:global-minor}.
\end{proof}
\section{The major arcs}\label{sec:major arcs}

We estimate the major arc form using its complete-sum and discrete
cut operators. The estimates hold for every integer $\ell\geq1$ with
$q\ell\leq N$. In Section~\ref{sec:density-increment} we choose $\ell$ as a
function of the denominator $q$ and the offset $\theta$.

\subsection{Reduction to the major arc form}

Fix integers $q,\ell\geq1$ and take $a\in(\Z/q\Z)^L$ and $\theta\in\R^L$.
Let $\mu_{q,\ell}$ be the probability measure on $\Z^d$ given by
\[
 \mu_{q,\ell}(x):=\ell^{-d}\1_{q\{0,\ldots,\ell-1\}^d}(x).
\]
For each $u_i\in\Z^d$ and $r_i\in(\Z/q\Z)^d$, let $y_i$ be the unique
point of $u_i-\{0,\ldots,q-1\}^d$ congruent to $r_i$ modulo $q$.
For finitely supported $F_0,\ldots,F_{K-1}:\Z^d\to\C$, define
\begin{align}
 &\mathcal M_{q,a,\theta}^{(\ell)}(F_0,\ldots,F_{K-1})\notag\\
 &\qquad:=\E_{r\bmod q}\sum_{u\in(\Z^d)^K}
 e_q\bigl(\mathcal Q_a(r)\bigr)e\bigl(\mathcal Q_\theta(u)\bigr)
 \prod_{i=0}^{K-1}(\mu_{q,\ell}*F_i)(y_i).
 \label{eq:major arc-form}
\end{align}
Here $r=(r_0,\ldots,r_{K-1})$ ranges over $((\Z/q\Z)^d)^K$ with normalized
counting measure, and $u=(u_0,\ldots,u_{K-1})$ ranges over $(\Z^d)^K$
with counting measure. The factor contributed by $F_i$
depends on $(u_i,r_i)$ through $y_i$ and is given by
\begin{equation}\label{eq:record-major-input}
 (\mu_{q,\ell}*F_i)(y_i)
 =\frac1{\ell^d}\sum_{h_i\in\{0,\ldots,\ell-1\}^d}F_i(y_i-qh_i).
\end{equation}

For each $y_i\in\Z^d$, the residue $r_i$ is determined by $y_i$ modulo $q$.
There are exactly $q^d$ choices of $u_i$, since
$u_i=y_i+s_i$ for $s_i\in\{0,\ldots,q-1\}^d$. These choices cancel the
factor $q^{-d}$ in the residue average, giving
\begin{equation}\label{eq:record-lift-identities}
 \E_{r_i\bmod q}\sum_{u_i\in\Z^d}|(\mu_{q,\ell}*F_i)(y_i)|^2
 =\sum_{y_i\in\Z^d}|(\mu_{q,\ell}*F_i)(y_i)|^2
 =\|\mu_{q,\ell}*F_i\|_{\ell^2(\Z^d)}^2.
\end{equation}
If $F_i$ is supported on $[N]^d$, a nonzero summand in
\eqref{eq:record-major-input} requires $y_i=z_i+qh_i$ for some
$z_i\in[N]^d$. Hence $u_i=z_i+qh_i+s_i\in[N+q\ell]^d$ whenever
$(\mu_{q,\ell}*F_i)(y_i)\ne0$.

\begin{lemma}\label{lem:major arc-approximation}
Take integers $q,\ell\geq1$ with $q\ell\leq N$ and
$a\in(\Z/q\Z)^L$ and $\theta\in\R^L$. For all
$F_0,\ldots,F_{K-1}:\Z^d\to\C$ supported on $[N]^d$, we have
\begin{align}
 &\left|\mathcal S_{a/q+\theta}(F_0,\ldots,F_{K-1})
 -\mathcal M_{q,a,\theta}^{(\ell)}(F_0,\ldots,F_{K-1})\right|\notag\\
 &\qquad\ll_{K,d}N\|\theta\|_\infty q\ell
 \prod_{i=0}^{K-1}\|F_i\|_{\ell^1(\Z^d)}.
 \label{eq:record-L1-approximation}
\end{align}
\end{lemma}

\begin{proof}
Expand the convolutions in \eqref{eq:major arc-form} using
\eqref{eq:record-major-input}, and write $u_i=y_i+s_i$ and $z_i=y_i-qh_i$. The residue condition becomes $z_i\equiv r_i\pmod q$, so
$e_q(\mathcal Q_a(r))=e_q(\mathcal Q_a(z))$. The map
$(h_i,s_i)\longmapsto t_i:=qh_i+s_i$
is a bijection from
$\{0,\ldots,\ell-1\}^d\times\{0,\ldots,q-1\}^d$ onto
$\{0,\ldots,q\ell-1\}^d$. Consequently,
\begin{align}
 \mathcal M_{q,a,\theta}^{(\ell)}(F_0,\ldots,F_{K-1})=\sum_{z\in(\Z^d)^K}
 e_q\bigl(\mathcal Q_a(z)\bigr)
 \prod_{i=0}^{K-1}F_i(z_i)\,
 \E_{t\in(\{0,\ldots,q\ell-1\}^d)^K}
 e\bigl(\mathcal Q_\theta(z+t)\bigr).
 \label{eq:major-approx-form}
\end{align}
The corresponding expression for
$\mathcal S_{a/q+\theta}(F_0,\ldots,F_{K-1})$ has the phase
$e(\mathcal Q_\theta(z))$ in place of the average over $t$ in \eqref{eq:major-approx-form}:
 \[\mathcal S_{a/q+\theta}(F_0,\ldots,F_{K-1})=\sum_{z\in(\Z^d)^K}
 e_q\bigl(\mathcal Q_a(z)\bigr)
 \prod_{i=0}^{K-1}F_i(z_i)\,
e(\mathcal Q_\theta(z)).\]
On the segment from $z$ to $z+t$, the assumption $q\ell\leq N$
gives $|x_i-x_j|\ll_d N$, and hence
\[
 \nabla_{x_i}\mathcal Q_\theta(x)
 =2\sum_{j:\{i,j\}\in\mathcal E}\theta_{ij}(x_i-x_j)
 \ll_{K,d}N\|\theta\|_\infty.
\]
The mean value theorem therefore yields the bound 
\[\left|e(\mathcal Q_\theta(z))-e\bigl(\mathcal Q_\theta(z+t)\bigr)\right|\ll_{K,d}N\|\theta\|_\infty q\ell.\] Averaging in $t$
and summing yields the desired result.
\end{proof}

\subsection{A pointwise major arc estimate}

Fix integers $q,\ell\geq1$ and $P\geq N+q\ell$, and take
$a\in(\Z/q\Z)^L$ and $\theta\in\R^L$. Suppose that
$F_0,\ldots,F_{K-1}$ are supported on $[N]^d$. The sums in
\eqref{eq:major arc-form} may then be restricted to $u_i\in[P]^d$
for every $i$.

Fix a balanced cut $S\mid T$. For each $i$, take
$X_i=[P]^d\times(\Z/q\Z)^d$.  Use the three phase pieces
\[
\begin{aligned}
  \mathcal Q^S(u_S,r_S)
  &:=
  \mathcal Q_\theta^S(u_S)+\frac{\mathcal Q_a^S(r_S)}q,\\
  \mathcal Q^{S,T}(u_S,r_S,u_T,r_T)
  &:=
  \mathcal Q_\theta^{S,T}(u_S,u_T)
  +\frac{\mathcal Q_a^{S,T}(r_S,r_T)}q,\\
  \mathcal Q^T(u_T,r_T)
  &:=
  \mathcal Q_\theta^T(u_T)+\frac{\mathcal Q_a^T(r_T)}q.
\end{aligned}
\]
The cut operator in
Lemma~\ref{lem:cut-cauchy-schwarz} is therefore
\[
\begin{aligned}
  \mathcal K^{S,T}:
  L^2\bigl(([P]^d\times(\Z/q\Z)^d)^T\bigr)
  &\longrightarrow
  L^2\bigl(([P]^d\times(\Z/q\Z)^d)^S\bigr),\\
  \bigl(\mathcal K^{S,T}H\bigr)(u_S,r_S)
  &:=
  \E_{r_T\bmod q}\sum_{u_T\in([P]^d)^T}
  e_q\!\left(\mathcal Q_a^{S,T}(r_S,r_T)\right)
  e\!\left(\mathcal Q_\theta^{S,T}(u_S,u_T)\right)
  H(u_T,r_T).
\end{aligned}
\]
By Lemma~\ref{lem:cut-cauchy-schwarz} and \eqref{eq:record-lift-identities},
\[
\abs{\mathcal M_{q,a,\theta}^{(\ell)}(F_0,\ldots,F_{K-1})}
\leq
\norm{\mathcal K^{S,T}}_{2\to2}
\prod_{i=0}^{K-1}\norm{\mu_{q,\ell}*F_i}_{\ell^2(\Z^d)}.
\]

The product structure of the kernel gives
$\mathcal K^{S,T}=\mathcal K_{P,d}^{S,T}(-2\theta_{S,T})\otimes\mathcal K_{q,a,d}^{S,T}$.

With normalized counting measure, define the complete cut operator
\[
\begin{aligned}
  \mathcal K_{q,a,d}^{S,T}:
  L^2\bigl(((\Z/q\Z)^d)^T\bigr)
  &\longrightarrow
  L^2\bigl(((\Z/q\Z)^d)^S\bigr),\\
  \bigl(\mathcal K_{q,a,d}^{S,T}H\bigr)(r_S)
  &:={}
  \E_{r_T\bmod q}
  e_q\!\left(
  -2\sum_{i\in S,\,j\in T}a_{ij}r_i\cdot r_j
  \right)H(r_T).
\end{aligned}
\]
We now compute the $2\to2$ norm of this cut operator.

\begin{lemma}
\label{lem:finite-bilinear-operator}
With normalized counting measure, $\norm{\mathcal K_{q,a,d}^{S,T}}_{2\to2}
=\abs{\operatorname{im}(2a_{S,T})}^{-d/2}$.
\end{lemma}

\begin{proof}
By the same argument as in the proof of
Lemma~\ref{lem:cut-tensorization}, $\norm{\mathcal K_{q,a,d}^{S,T}}_{2\to2}
=\norm{\mathcal K_{q,a,1}^{S,T}}_{2\to2}^{d}$. It therefore suffices to compute
the one-dimensional norm.

Write $\mathcal K:=\mathcal K_{q,a,1}^{S,T}$. As in the proof of
Lemma~\ref{lem:rectangular-moment}, we sum over the $T$-variables in
$\mathcal K\mathcal K^*$. For a function $f$ on $(\Z/q\Z)^S$, expanding
$\mathcal K\mathcal K^*$ gives
\[
\begin{aligned}
  (\mathcal K\mathcal K^*f)(r_S)
  &=
  \E_{r_S'\bmod q}f(r_S')
  \E_{r_T\bmod q}
  e_q\!\left(
  (r_S'-r_S)^{\mathsf T}2a_{S,T}r_T
  \right)\\
  &=
  q^{-\abs S}
  \sum_{\substack{r_S'\bmod q\\
      (2a_{S,T})^{\mathsf T}(r_S-r_S')=0}}
  f(r_S'),
\end{aligned}
\]
where the second equality follows from the orthogonality relation.
Making the substitution $h=r_S-r_S'$ gives
\[
(\mathcal K\mathcal K^*f)(r_S)
=
q^{-\abs S}
\sum_{h\in\ker((2a_{S,T})^{\mathsf T})}
f(r_S-h).
\]
All translates of $f$ have the same $L^2$ norm. By the triangle inequality,
\[
 \norm{\mathcal K\mathcal K^*f}_{L^2((\Z/q\Z)^S)}
 \leq\frac{\abs{\ker((2a_{S,T})^{\mathsf T})}}{q^{\abs S}}
 \norm{f}_{L^2((\Z/q\Z)^S)}.
\]
This bound is attained by $f=\mathbf 1$, so
$\norm{\mathcal K\mathcal K^*}_{2\to2}
=\abs{\ker((2a_{S,T})^{\mathsf T})}/q^{\abs S}$.

Finally, the orthogonality relation gives
\[
\begin{aligned}
  \abs{\ker((2a_{S,T})^{\mathsf T})}
  &=
  \sum_{r_S\bmod q}
  \E_{r_T\bmod q}
  e_q\!\left(r_S^{\mathsf T}2a_{S,T}r_T\right)\\
  &=
  q^{\abs S-\abs T}\abs{\ker(2a_{S,T})}\\
  &=
  \frac{q^{\abs S}}
  {\abs{\operatorname{im}(2a_{S,T})}},
\end{aligned}
\]
where the last equality follows from
$q^{\abs T} = \abs{\ker(2a_{S,T})} \abs{\operatorname{im}(2a_{S,T})}$.
Consequently, $\norm{\mathcal K\mathcal K^*}_{2\to2}
=\abs{\operatorname{im}(2a_{S,T})}^{-1}$, so
$\norm{\mathcal K_{q,a,1}^{S,T}}_{2\to2}=\abs{\operatorname{im}(2a_{S,T})}^{-1/2}$.
Raising this identity to the $d$th power gives
$\norm{\mathcal K_{q,a,d}^{S,T}}_{2\to2}=\abs{\operatorname{im}(2a_{S,T})}^{-d/2}$.
\end{proof}

We use the exponent
$\tau:=d\lfloor K^2/4\rfloor/(2L)$
introduced in the proof of Proposition~\ref{prop:minor arcs}.
Since $d/\abs{\mathscr B_K}=2\tau w_K$, the geometric mean of the
complete cut operator norms over the balanced cuts is
\begin{equation}\label{eq:complete-sum-majorant}
  \mathfrak s_q(a)
  :=
  \prod_{S\mid T\in\mathscr B_K}
  \abs{\operatorname{im}(2a_{S,T})}^{-\tau w_K}.
\end{equation}

\begin{proposition}\label{prop:smoothed-major-bound}
Let $q,\ell\geq1$ be integers with $q\ell\leq N$.
Let $F_0,\ldots,F_{K-1}:\Z^d\to\C$ be supported on $[N]^d$.
For every $a\in(\Z/q\Z)^L$ and $\theta\in\R^L$, we have
\begin{align}
 \abs{\mathcal S_{a/q+\theta}(F_0,\ldots,F_{K-1})}
 &\ll_{K,d} N^{Kd/2}\mathfrak s_q(a)I_{2N}(\theta)
 \prod_{i=0}^{K-1}\norm{\mu_{q,\ell}*F_i}_{\ell^2(\Z^d)} \notag\\
 &\quad{}+N\|\theta\|_\infty q\ell
 \prod_{i=0}^{K-1}\norm{F_i}_{\ell^1(\Z^d)}.
 \label{eq:smoothed-major-bound}
\end{align}
\end{proposition}

\begin{proof}
Put $P=2N$. Since $q\ell\leq N$, we have $N+q\ell\leq P$.
Since $\mathcal K^{S,T}=\mathcal K_{P,d}^{S,T}(-2\theta_{S,T})\otimes\mathcal K_{q,a,d}^{S,T}$,
multiplicativity of the operator norm and Lemmas~\ref{lem:finite-bilinear-operator}
and \ref{lem:cut-tensorization} give
\[
\begin{aligned}
  \norm{\mathcal K^{S,T}}_{2\to2}
  =
  \norm{
    \mathcal K_{P,d}^{S,T}\!\left(-2\theta_{S,T}\right)
  }_{2\to2}
  \norm{\mathcal K_{q,a,d}^{S,T}}_{2\to2}=
  P^{Kd/2}
  m_P^{S,T}\!\left(-2\theta_{S,T}\right)^d
  \abs{\operatorname{im}(2a_{S,T})}^{-d/2}.
\end{aligned}
\]

Taking the geometric mean over the balanced cuts and using
\eqref{eq:balanced-cut-mean} and \eqref{eq:complete-sum-majorant}, we obtain
\begin{align}
  &\abs{\mathcal M_{q,a,\theta}^{(\ell)}(F_0,\ldots,F_{K-1})}\notag\\
  &\quad\leq P^{Kd/2}
  \prod_{S\mid T\in\mathscr B_K}
  \left[\abs{\operatorname{im}(2a_{S,T})}^{-1/2}
  m_P^{S,T}(-2\theta_{S,T})\right]^{2\tau w_K}
  \prod_{i=0}^{K-1}\norm{\mu_{q,\ell}*F_i}_{\ell^2(\Z^d)}\notag\\
  &\quad= P^{Kd/2}\mathfrak s_q(a)I_P(\theta)
  \prod_{i=0}^{K-1}\norm{\mu_{q,\ell}*F_i}_{\ell^2(\Z^d)}.
  \label{eq:major-pointwise-cut-bound}
\end{align}

Substituting $P=2N$ into
\eqref{eq:major-pointwise-cut-bound} and adding the approximation
error from Lemma~\ref{lem:major arc-approximation} gives
\eqref{eq:smoothed-major-bound}.
\end{proof}

\subsection{The complete sum estimate}

We now estimate the sum of $\mathfrak s_q(a)$ over
numerators $a\bmod q$ with $(a,q)=1$. 

\begin{lemma}
Let $p$ be prime and let $m,s,t$ be positive integers. Take normalized
averages over $\mathsf M\in\operatorname{Mat}_{s\times t}(\Z/p^m\Z)$.
If $\tau>s+t-1$, then for every fixed pair of indices
$1\leq i\leq s$ and $1\leq j\leq t$, we have
\begin{align}
  \E_{\mathsf M}\abs{\operatorname{im}(2\mathsf M)}^{-\tau}
  &\ll_{s,t,\tau} p^{-mst},
  \label{eq:finite-matrix-unconditioned}\\
  \E_{\mathsf M}\!\left[
  \1_{\mathsf M_{ij}\in(\Z/p^m\Z)^\times}
  \abs{\operatorname{im}(2\mathsf M)}^{-\tau}\right]
  &\ll_{s,t,\tau} p^{-m[\tau+(s-1)(t-1)]}.
  \label{eq:finite-matrix-unit}
\end{align}
The implicit constants depend only on $s,t,$ and $\tau$.
\end{lemma}

\begin{proof}
Put
\[
W_{s,t}(p^m)
:=
\sum_{\mathsf M\in\operatorname{Mat}_{s\times t}(\Z/p^m\Z)}
\abs{\operatorname{im}\mathsf M}^{-\tau},
\]
with the conventions
$W_{0,t}(p^m)=W_{s,0}(p^m)=W_{s,t}(1)=1$.

The matrices divisible by $p$ contribute $W_{s,t}(p^{m-1})$.  Every
other matrix has a unit entry.  Choose such an entry and use it as a
pivot. After permuting rows and columns, write
\[
\mathsf M
=
\begin{pmatrix}
  u&\mathsf b\\
  \mathsf c&\mathsf D
\end{pmatrix},
\]
where $\mathsf b$ is a row vector and $\mathsf c$ is a column vector.
By applying row and column operations we reduce $\mathsf M$ to
\[
\begin{pmatrix}
  u&0\\
  0&\mathsf M'
\end{pmatrix},
\]
where $\mathsf M'=\mathsf D-\mathsf c u^{-1}\mathsf b$.
These operations preserve the size of the image, and hence
$\abs{\operatorname{im}\mathsf M} = p^m\abs{\operatorname{im}\mathsf M'}.$
There are $O_{s,t}(p^{m(s+t-1)})$ choices for the position of the pivot
and for $u,\mathsf b,$ and $\mathsf c$.  For each such choice, the map
$\mathsf D\longmapsto \mathsf M'=\mathsf D-\mathsf c u^{-1}\mathsf b$
is a bijection of
$\operatorname{Mat}_{(s-1)\times(t-1)}(\Z/p^m\Z)$, so $\mathsf M'$ may
be summed.  Therefore
\[
W_{s,t}(p^m)
\leq
W_{s,t}(p^{m-1})
+
c_{s,t}p^{-m(\tau-s-t+1)}
W_{s-1,t-1}(p^m).
\]

We now induct on $\min(s,t)$.  The base case is $\min(s,t)=0$, when
$W_{s,t}(p^m)=1$ by convention.  Suppose that $s,t\geq1$, and assume
inductively that
$W_{s-1,t-1}(p^r)\ll_{s,t,\tau}1$
uniformly in $p$ and $r$.  Set $\delta:=\tau-s-t+1>0$.  Iterating the
preceding recursion from $m$ down to $0$ gives
\[
\begin{aligned}
  W_{s,t}(p^m)
  &\leq
  W_{s,t}(p^{m-1})
  +
  c_{s,t}p^{-m\delta}W_{s-1,t-1}(p^m)\\
  &\leq
  W_{s,t}(p^{m-2})
  +
  c_{s,t}p^{-(m-1)\delta}W_{s-1,t-1}(p^{m-1})
  +
  c_{s,t}p^{-m\delta}W_{s-1,t-1}(p^m)\\
  &\ \vdots\\
  &\leq
  W_{s,t}(1)
  +
  c_{s,t}\sum_{r=1}^m
  p^{-r\delta}W_{s-1,t-1}(p^r)\\
  &\ll
  1+\sum_{r=1}^m p^{-r\delta}
  \ll 1.
\end{aligned}
\]
Consequently,
\begin{equation}\label{7.13}
  \E_{\mathsf M}
  \abs{\operatorname{im}\mathsf M}^{-\tau}
  =
  p^{-mst}W_{s,t}(p^m)
  \ll
  p^{-mst}.
\end{equation}

It remains to pass from $\mathsf M$ to $2\mathsf M$. Multiplication by $2$ maps
$\operatorname{im}\mathsf M$ onto $\operatorname{im}(2\mathsf M)$.
Each coordinate of a vector in its kernel satisfies $2x=0\bmod p^m$,
which has at most two solutions. Thus
$\abs{\operatorname{im}\mathsf M}\leq2^s\abs{\operatorname{im}(2\mathsf M)}$,
and \eqref{7.13} gives \eqref{eq:finite-matrix-unconditioned}.

For \eqref{eq:finite-matrix-unit}, use the prescribed unit entry
$\mathsf M_{ij}$ as the pivot. As above, there are at most $p^{m(s+t-1)}$
choices for its row and column, and
$\abs{\operatorname{im}\mathsf M}=p^m\abs{\operatorname{im}\mathsf M'}$.
Consequently,
\begin{align*}
 \E_{\mathsf M}\!\left[\1_{\mathsf M_{ij}\in(\Z/p^m\Z)^\times}
          \abs{\operatorname{im}(2\mathsf M)}^{-\tau}\right]
 &\leq 2^{s\tau}p^{-mst}p^{m(s+t-1)}p^{-m\tau}W_{s-1,t-1}(p^m)\\
 &\ll_{s,t,\tau}p^{-m[\tau+(s-1)(t-1)]}.
\end{align*}
\end{proof}

Under the hypothesis $d\geq4K+4$, the proof of
Proposition~\ref{prop:minor arcs} shows that $\tau>K+2$.
Then put
$\eta:=(\tau-K)/2.$
Since $\tau>K+2$, we have $\eta>1$.

\begin{lemma}
\label{lem:modular-numerator-sum}
If $d\geq4K+4$, then $\eta>1$ and, for every $q\geq1$,
\begin{equation}\label{eq:modular-numerator-bound}
  \sum_{\substack{a\in(\Z/q\Z)^L\\(a,q)=1}}
  \mathfrak s_q(a)
  \ll_{K,d}q^{-1-\eta}.
\end{equation}
\end{lemma}	
\begin{proof}
The case $q=1$ is immediate.  First let $q=p^m$.  For each
$e\in\mathcal E$, put
\[
I_e
:=
\E_{a\in(\Z/p^m\Z)^L}\!\left[
\1_{a_e\in(\Z/p^m\Z)^\times}
\mathfrak s_{p^m}(a)
\right].
\]

Fix $e\in\mathcal E$. When applying Lemma~\ref{lem:fractional-holder},
take $X_{e'}=\Z/p^m\Z$ with normalized counting measure for each $e'\in\mathcal E$.
Take $J=\abs{\mathscr B_K}$ and assign an index in  $\{1,\ldots,J\}$ to each balanced cut. Thus, for the index $j$ corresponding to
$S\mid T$, the set $\mathcal E_j$ in that lemma is the set of edges
crossing $S\mid T$, and the weight $c_j$ is $w_K$. We write
$f_{S,T}$ for the corresponding function $f_j$.

For each balanced cut $S\mid T$, define
\[
f_{S,T}(a_{S,T})
:=
\begin{cases}
  \displaystyle
  \1_{a_e\in(\Z/p^m\Z)^\times}
  \abs{\operatorname{im}(2a_{S,T})}^{-\tau},
  & e\text{ crosses }S\mid T,\\[6pt]
  \displaystyle
  \abs{\operatorname{im}(2a_{S,T})}^{-\tau},
  & e\text{ does not cross }S\mid T.
\end{cases}
\]
When $e$ crosses $S\mid T$, the coordinate $a_e$ is one of the
coordinates in $a_{S,T}$, so the first function is indeed defined
on the product of the spaces indexed by $\mathcal E_j$.

For every $e'\in\mathcal E$, Lemma~\ref{lem:balanced-cover} gives
\[
\sum_{\substack{S\mid T\in\mathscr B_K\\
    e'\text{ crosses }S\mid T}}
w_K
=
1,
\]
which is exactly the hypothesis of
Lemma~\ref{lem:fractional-holder}. Moreover,
\begin{align*}
  \prod_{S\mid T\in\mathscr B_K}
  f_{S,T}(a_{S,T})^{w_K}
  ={}&
  \left(
  \prod_{\substack{S\mid T\in\mathscr B_K\\
      e\text{ crosses }S\mid T}}
  \1_{a_e\in(\Z/p^m\Z)^\times}^{\,w_K}
  \right)
  \prod_{S\mid T\in\mathscr B_K}
  \abs{\operatorname{im}(2a_{S,T})}^{-\tau w_K}=
  \1_{a_e\in(\Z/p^m\Z)^\times}
  \mathfrak s_{p^m}(a),
\end{align*}
where the last equality follows from
\eqref{eq:fractional-cover}. Hence
\[
I_e
=
\E_{a\in(\Z/p^m\Z)^L}
\prod_{S\mid T\in\mathscr B_K}
f_{S,T}(a_{S,T})^{w_K}.
\]
Lemma~\ref{lem:fractional-holder} therefore gives
\begin{align*}
  I_e
  \leq{}&
  \prod_{\substack{S\mid T\in\mathscr B_K\\
      e\text{ crosses }S\mid T}}
  \left(
  \E_{a_{S,T}}
  \left[
  \1_{a_e\in(\Z/p^m\Z)^\times}
  \abs{\operatorname{im}(2a_{S,T})}^{-\tau}
  \right]
  \right)^{w_K}\\
  &\times
  \prod_{\substack{S\mid T\in\mathscr B_K\\
      e\text{ does not cross }S\mid T}}
  \left(
  \E_{a_{S,T}}
  \abs{\operatorname{im}(2a_{S,T})}^{-\tau}
  \right)^{w_K},
\end{align*}
where each average in $a_{S,T}$ is normalized over
$\operatorname{Mat}_{\abs S\times\abs T}(\Z/p^m\Z)$.

We now estimate the averages on the right. Fix a balanced cut
$S\mid T$. If $e$ crosses $S\mid T$, then
\eqref{eq:finite-matrix-unit} gives
\begin{align*}
  \E_{a_{S,T}}
  \left[
  \1_{a_e\in(\Z/p^m\Z)^\times}
  \abs{\operatorname{im}(2a_{S,T})}^{-\tau}
  \right]\ll_{K,d}
  p^{-m[\tau+(\abs S-1)(\abs T-1)]}
  =
  p^{-m(\abs S\abs T+1+2\eta)},
\end{align*}
where we used $\abs S+\abs T=K$ and
$2\eta=\tau-K$. If $e$ does not cross $S\mid T$, then
\eqref{eq:finite-matrix-unconditioned} gives
\[
\E_{a_{S,T}}
\abs{\operatorname{im}(2a_{S,T})}^{-\tau}
\ll_{K,d}
p^{-m\abs S\abs T}.
\]
Consequently,
\[
I_e
\ll_{K,d}
p^{-mw_K\sum_{S\mid T\in\mathscr B_K}\abs S\abs T}
p^{-m(1+2\eta)w_K
  \#\{S\mid T\in\mathscr B_K:e\text{ crosses }S\mid T\}}.
\]
Since $w_K\sum_{S\mid T\in\mathscr B_K}\abs S\abs T=L$,
\eqref{eq:fractional-cover} gives $I_e\ll_{K,d}p^{-m(L+1+2\eta)}$.

If $(a,p^m)=1$, then $a_e$ is a unit for at least one
$e\in\mathcal E$. Therefore
\[
\1_{(a,p^m)=1}
\leq
\sum_{e\in\mathcal E}
\1_{a_e\in(\Z/p^m\Z)^\times},
\]
and hence
\begin{align*}
  \sum_{\substack{a\in(\Z/p^m\Z)^L\\(a,p^m)=1}}
  \mathfrak s_{p^m}(a)
  =
  p^{mL}
  \E_{a\in(\Z/p^m\Z)^L}\!\left[
  \1_{(a,p^m)=1}\mathfrak s_{p^m}(a)
  \right]\leq
  p^{mL}\sum_{e\in\mathcal E}I_e
  \ll_{K,d}
  p^{-m(1+2\eta)}.
\end{align*}

Now write $q=\prod_{r=1}^u p_r^{m_r}.$
Under the Chinese remainder theorem,
$(\Z/q\Z)^L \cong \prod_{r=1}^u(\Z/p_r^{m_r}\Z)^L.$
Write $a^{(r)}$ for the component of $a$ modulo $p_r^{m_r}$.
Then $(a,q)=1$ if and only if $(a^{(r)},p_r^{m_r})=1$ for every
$1\leq r\leq u$.
Moreover, for every balanced cut $S\mid T$, the map
$2a_{S,T}$ corresponds under the Chinese remainder theorem to the
product of the maps $2a_{S,T}^{(r)}$. Therefore
\[
\abs{\operatorname{im}(2a_{S,T})}
=
\prod_{r=1}^u
\abs{\operatorname{im}(2a_{S,T}^{(r)})},
\]
and consequently
$\mathfrak s_q(a) = \prod_{r=1}^u \mathfrak s_{p_r^{m_r}}(a^{(r)}).$
It follows that
\begin{align*}
  \sum_{\substack{a\in(\Z/q\Z)^L\\(a,q)=1}}
  \mathfrak s_q(a)
  =
  \prod_{r=1}^u
  \left(
  \sum_{\substack{
      a^{(r)}\in(\Z/p_r^{m_r}\Z)^L\\
      (a^{(r)},p_r^{m_r})=1}}
  \mathfrak s_{p_r^{m_r}}(a^{(r)})
  \right)\ll_{K,d}
  C^{\omega(q)}
  \prod_{r=1}^u p_r^{-m_r(1+2\eta)}=
  C^{\omega(q)}q^{-1-2\eta}
\end{align*}
for some $C=C(K,d)$. Since
$C^{\omega(q)}\ll_{C,\eta}q^\eta$, this proves
\eqref{eq:modular-numerator-bound}.
\end{proof}

\subsection{Summing the major arc estimate}

We use $\gamma=\min\{1,\tau-K-2\}$ as in the proof of Proposition~\ref{prop:minor arcs}.

\begin{lemma}\label{lem:weighted-major-moment}
Assume $d\geq4K+4$. For every integer $P\geq8$,
\begin{equation}\label{eq:weighted-major-moment}
 \sum_{q\geq1}\sum_{\substack{a\bmod q\\(a,q)=1}}
 \mathfrak s_q(a)
 \int_{\|\theta\|_\infty\leq1/(8P)}
 [q(1+P^2\|\theta\|_\infty)]^\gamma
 I_P(\theta)\,d\theta
 \ll_{K,d}P^{-2L}.
\end{equation}
Here $\theta$ ranges over the indicated box in $\R^L$.
\end{lemma}

\begin{proof}
Suppose $0<|\theta_e|\leq1/(8P)$. The points $-2r\theta_e\in\T$, $1\leq r\leq P$,
are $2|\theta_e|$-separated, since $2\abs{r-r'}\abs{\theta_e}<1/2$ for distinct
$r,r'\in[P]$. Applying the analytic large sieve
\cite[Theorem~1, equation~(1.4)]{MontgomeryVaughan} at these points gives
\[
 m_P(-2\theta_e)^2\leq P^{-2}\bigl(P+(2|\theta_e|)^{-1}\bigr)
 \ll (P^2|\theta_e|)^{-1}.
\]
Together with $m_P(-2\theta_e)\leq1$, this proves
\begin{equation}\label{eq:major-scalar-decay}
 m_P(-2\theta_e)\ll(1+P^2|\theta_e|)^{-1/2}
\end{equation}
for all $|\theta_e|\leq1/(8P)$, including $\theta_e=0$.

Fix $\theta$ in the box of integration and choose an edge $e$ for which
$|\theta_e|=\|\theta\|_\infty$. By \eqref{eq:major-scalar-decay},
$(1+P^2\|\theta\|_\infty)^\gamma m_P(-2\theta_e)^{2\gamma}\ll_{K,d}1$.
Using this in \eqref{eq:edge-extraction}, enlarging the box of integration
to $\T^L$ on the right, and applying \eqref{eq:full-cut-moment} gives
\begin{align*}
 \int_{\|\theta\|_\infty\leq1/(8P)}
 (1+P^2\|\theta\|_\infty)^\gamma I_P(\theta)\,d\theta
 &\ll_{K,d}\int_{\T^L}\prod_{S\mid T\in\mathscr B_K}
 m_P^{S,T}(-2\theta_{S,T})^{(2K+4)w_K}\,d\theta\\
 &\ll_{K,d}P^{-2L}.
\end{align*}
Finally, Lemma~\ref{lem:modular-numerator-sum} gives
\[
 \sum_{q\geq1}q^\gamma
 \sum_{\substack{a\bmod q\\(a,q)=1}}\mathfrak s_q(a)
 \ll_{K,d}\sum_{q\geq1}q^{-1-\eta+\gamma}
 \ll_{K,d}1,
\]
since $\gamma\leq1<\eta$. Multiplying the last two estimates
proves \eqref{eq:weighted-major-moment}.
\end{proof}

\section{Induction on scale}
\label{sec:density-increment}

\begin{proof}[Proof of Theorem~\ref{thm:main}]
\begingroup
\setlength{\abovedisplayskip}{7pt plus 2pt minus 2pt}
We will choose a small constant $0<c<1/2$ and then a large constant $N_0\geq1$, both depending on $d$ and $\Delta$. 
Set $M=N_0^c$. Suppose that there is a set $A\subseteq[N]^d$
containing no nontrivial similar copy of $\Delta$ with
$\abs A>MN^{d-c}$, and choose the least possible $N$. Set
$\alpha:=\abs A/N^d$ and the balanced function $g:=\1_A-\alpha\1_{[N]^d}$. Then
$\alpha>MN^{-c}\geq N^{-c}$. Since $\alpha\leq1$, we have $N>M^{1/c}=N_0$
throughout the proof.

By Kitaoka's estimate (Lemma~\ref{Kitaoka}), there is a constant $b>0$ depending on $\Delta$ and $d$ such
that, if $N_0$ is chosen large enough,
\[
 \Lambda_N(\1_{[N]^d},\ldots,\1_{[N]^d})\geq bN^{Kd-2L},
\]
where we used $N>N_0$. 
Our goal is to show that
\begin{equation}\label{eq:induction-counting-comparison}
 \begin{aligned}
 &\abs{\Lambda_N(\1_A,\ldots,\1_A)
 -\alpha^K\Lambda_N(\1_{[N]^d},\ldots,\1_{[N]^d})}\leq\frac b2\alpha^KN^{Kd-2L}.
 \end{aligned}
\end{equation}
Indeed, \eqref{eq:induction-counting-comparison} would give $\Lambda_N(\1_A,\ldots,\1_A)
\geq \tfrac b2\alpha^KN^{Kd-2L}>0,$
contradicting the fact that $A$ was chosen to have no nondegenerate similar copies of $\Delta$.

We first use the minimality of $N$ to bound the variance of the relative density of $A$ on the $q$-spaced subcubes of a fixed sidelength. This variance occurs naturally in the major arc estimate
\eqref{eq:smoothed-major-bound} (through the factor $\|\mu_{q,\ell}\ast g\|_{\ell^2(\Z^d)}$). Fix integers $q\geq1$ and $1\leq\ell<N$
with $q\ell\leq N$. By minimality of $N$, for every $x\in\Z^d$ the cube
$x-q\{0,\ldots,\ell-1\}^d$ contains at most $M\ell^{d-c}$ points of $A$,
since translating and dividing by $q$ preserves
the lack of similar copies of $\Delta$.

Let $\mu_{q,\ell}$ be normalized counting measure on
$q\{0,\ldots,\ell-1\}^d$, as in Section~\ref{sec:major arcs}.
From our assumption $M<\alpha N^c$ we have 
\[
 0\leq(\mu_{q,\ell}*\1_A)(x)
 =\frac1{\ell^d}\sum_{u\in\{0,\ldots,\ell-1\}^d}\1_A(x-qu)
 \leq M\ell^{-c}\leq \alpha(N/\ell)^c.
\]
Since $\mu_{q,\ell}$ is a probability measure,
\begin{equation}\label{eq:induction-A-mass}
 \sum_{x\in\Z^d}(\mu_{q,\ell}*\1_A)(x)=\alpha N^d.
\end{equation}
This and the pointwise bound $ (\mu_{q,\ell}*\1_A)(x)
 \leq \alpha(N/\ell)^c$ yield
\begin{equation}\label{eq:induction-A-second-moment}
 \norm{\mu_{q,\ell}*\1_A}_{\ell^2(\Z^d)}^2
 \leq\alpha(N/\ell)^c
 \sum_{x\in\Z^d}(\mu_{q,\ell}*\1_A)(x)
 =\alpha^2(N/\ell)^cN^d.
\end{equation}
The convolution $\mu_{q,\ell}*\1_{[N]^d}$ takes values in $[0,1]$ and
has total mass $N^d$, so
\begin{equation}\label{eq:induction-box-second-moment}
 \norm{\mu_{q,\ell}*\1_{[N]^d}}_{\ell^2(\Z^d)}^2
 \leq \sum_{x\in\Z^d}(\mu_{q,\ell}*\1_{[N]^d})(x)=N^d.
\end{equation}

\endgroup
The convolutions $\mu_{q,\ell}*\1_A$ and $\mu_{q,\ell}*\1_{[N]^d}$
are supported on $[P]^d$, where $P:=N+q\ell$.
Expanding and using \eqref{eq:induction-A-mass} and
\eqref{eq:induction-A-second-moment}, we obtain
\begin{align*}
 \norm{\mu_{q,\ell}*\1_A-\alpha\1_{[P]^d}}_{\ell^2(\Z^d)}^2
 &=\sum_{x\in[P]^d}\bigl((\mu_{q,\ell}*\1_A)(x)-\alpha\bigr)^2\\
 &=\norm{\mu_{q,\ell}*\1_A}_{\ell^2(\Z^d)}^2
  -2\alpha\sum_{x\in[P]^d}(\mu_{q,\ell}*\1_A)(x)+\alpha^2P^d\\
 &\leq\alpha^2(N/\ell)^cN^d-2\alpha^2N^d+\alpha^2P^d\\
 &=\alpha^2N^d\bigl((N/\ell)^c-1\bigr)
  +\alpha^2(P^d-N^d).
\end{align*}
Similarly, using $(1-t)^2\leq 1-t$ for $0\leq t\leq1$ and
\eqref{eq:induction-box-second-moment}, we obtain
\begin{align*}
 \norm{\mu_{q,\ell}*\1_{[N]^d}-\1_{[P]^d}}_{\ell^2(\Z^d)}^2
 =\sum_{x\in[P]^d}\bigl(1-(\mu_{q,\ell}*\1_{[N]^d})(x)\bigr)^2\leq\sum_{x\in[P]^d}\bigl(1-(\mu_{q,\ell}*\1_{[N]^d})(x)\bigr)=\!P^d\!-\!N^d.
\end{align*}
Since $g=\1_A-\alpha\1_{[N]^d}$, the inequality
$(u-v)^2\leq 2u^2+2v^2$ and the two estimates above give
\begin{align}
 \norm{\mu_{q,\ell}*g}_{\ell^2(\Z^d)}^2
 &=\bigl\|(\mu_{q,\ell}*\1_A-\alpha\1_{[P]^d})
    -\alpha(\mu_{q,\ell}*\1_{[N]^d}-\1_{[P]^d})\bigr\|_{\ell^2(\Z^d)}^2 \notag\\
 &\leq 2\norm{\mu_{q,\ell}*\1_A-\alpha\1_{[P]^d}}_{\ell^2(\Z^d)}^2
   +2\alpha^2\norm{\mu_{q,\ell}*\1_{[N]^d}-\1_{[P]^d}}_{\ell^2(\Z^d)}^2 \notag\\
 &\leq 2\alpha^2N^d\bigl((N/\ell)^c-1\bigr)+4\alpha^2(P^d-N^d) \notag\\
 &\ll_d \alpha^2N^d\bigl((N/\ell)^c-1+q\ell/N\bigr).
 \label{eq:induction-variance}
\end{align}
In the last step we used $P^d-N^d=N^d\bigl((1+q\ell/N)^d-1\bigr)
\ll_d (q\ell/N)N^d$, since $0\leq q\ell/N\leq 1$.

Recall that the goal in \eqref{eq:induction-counting-comparison} is to
compare $\Lambda_N(\1_A,\ldots,\1_A)$ and
$\alpha^K\Lambda_N(\1_{[N]^d},\ldots,\1_{[N]^d})$.
Repeatedly making the substitution $\1_A=\alpha\1_{[N]^d}+g$ gives
\begin{align*}
 \Lambda_N(\1_A,\ldots,\1_A)
 &= \alpha\Lambda_N(\1_A,\ldots,\1_A,\1_{[N]^d})
  +\Lambda_N(\1_A,\ldots,\1_A,g)\\[2pt]
 &= \alpha^2\Lambda_N(\1_A,\ldots,\1_A,\1_{[N]^d},\1_{[N]^d})+\alpha\Lambda_N(\1_A,\ldots,\1_A,g,\1_{[N]^d})
  +\Lambda_N(\1_A,\ldots,\1_A,g)\\
 &= \cdots\\
 &= \alpha^K\Lambda_N(\1_{[N]^d},\ldots,\1_{[N]^d})+\sum_{j=0}^{K-1}\alpha^{K-1-j}
  \Lambda_N\bigl(\underbrace{\1_A,\ldots,\1_A}_{j},\,g,\,
   \underbrace{\1_{[N]^d},\ldots,\1_{[N]^d}}_{K-1-j}\bigr).
\end{align*}
The absolute value of the sum over $j$ is the left-hand side of 
\eqref{eq:induction-counting-comparison}. Fix $0\leq j\leq K-1$.
Including the factor $\alpha^{K-1-j}$ in the last $K-1-j$ inputs,
we write the $j$-th summand as $\Lambda_N(F_0,\ldots,F_{K-1})$, where
\[
 (F_0,\ldots,F_{K-1})=
 \bigl(\underbrace{\1_A,\ldots,\1_A}_{j},\,g,\,
 \underbrace{\alpha\1_{[N]^d},\ldots,\alpha\1_{[N]^d}}_{K-1-j}\bigr).
\]
By \eqref{eq:arc-restricted-form},
\[
 \Lambda_N(F_0,\ldots,F_{K-1})
 =\int_{\T^L}\widehat w_R\bigl(\mathcal Q_\beta(v)\bigr)
 \mathcal S_\beta(F_0,\ldots,F_{K-1})\,d\beta.
\]
Fix a sufficiently small $\kappa>0$ such that $\kappa(L+3)<1$,
and put $Q:=N^\kappa$. Take $N_0$ large enough that $2\leq Q\leq R/2$.
The contribution of the minor arcs $\mathfrak m_Q$ to this integral is
$\Lambda_N^{\mathfrak m_Q}(F_0,\ldots,F_{K-1})$.
Since $\|g\|_{\ell^2(\Z^d)}^2=\alpha(1-\alpha)N^d$ and
$\|\alpha\1_{[N]^d}\|_{\ell^2(\Z^d)}=\alpha N^{d/2}
\leq\|\1_A\|_{\ell^2(\Z^d)}=\alpha^{1/2}N^{d/2}$,
Proposition~\ref{prop:minor arcs} and $R\asymp_\Delta N$ give
\[
 \abs{\Lambda_N^{\mathfrak m_Q}(F_0,\ldots,F_{K-1})}
 \ll_{\Delta,d}Q^{-\gamma/L}\alpha^{K/2}N^{Kd-2L}.
\]
Dividing by $\alpha^K N^{Kd-2L}$ and using $\alpha>N^{-c}$, we have
\[
 \frac{\abs{\Lambda_N^{\mathfrak m_Q}(F_0,\ldots,F_{K-1})}}{\alpha^K N^{Kd-2L}}
 \ll_{\Delta,d}Q^{-\gamma/L}\alpha^{-K/2}
 \leq N^{-\kappa\gamma/L+cK/2}=o_{N\to\infty}(1)
\]
provided $c<2\kappa\gamma/(KL)$. Thus, for $N_0$ sufficiently large,
the minor arc contribution for each $j$ is $\leq\tfrac{b}{4K}\alpha^KN^{Kd-2L}$. Since $|\widehat w_R|\leq1$,
it remains to show that
\[
 \int_{\mathfrak M_Q}\abs{\mathcal S_\beta(F_0,\ldots,F_{K-1})}\,d\beta
 \leq\frac{b}{4K}\alpha^KN^{Kd-2L}.
\]

Recall that $Q=N^\kappa$, where $0<\kappa<1/(L+3)$.
Suppose $1\leq q\leq Q$ and $\theta\in\R^L$ with
$\norm{\theta}_\infty\leq Q/(qR^2)$.
Suppose further that $a\in(\Z/q\Z)^L$ with $(a,q)=1$ and $\beta=a/q+\theta$.
We choose the sidelength of the subcube in terms of $q$ and $\theta$:
$\ell=\ell(q,\theta):=\lfloor cN/t^{L+3}\rfloor$, where
$t:=q(1+N^2\norm{\theta}_\infty)$.
Since $R\asymp_\Delta N$, we have
$1\leq q\leq t\leq q+N^2Q/R^2\ll_\Delta Q$, and hence
\[
 \ell\geq \frac{cN}{2t^{L+3}}
 \gg_\Delta cN^{1-\kappa(L+3)}
 \longrightarrow\infty
\]
as $N\to\infty$.
Moreover,
$q\ell/N\leq cq/t^{L+3}\leq c$.

Since $q\ell\leq N$ and
$\|\1_A\|_{\ell^1(\Z^d)}=\alpha\|\1_{[N]^d}\|_{\ell^1(\Z^d)}=\alpha N^d$,
applying \eqref{eq:smoothed-major-bound} gives
\[
 \begin{aligned}
 &\abs{\mathcal S_\beta(F_0,\ldots,F_{K-1})}\\
 &\quad\ll_{K,d}
 \begin{aligned}[t]
 &N^{Kd/2}\mathfrak s_q(a)
 I_{2N}(\theta)
 \|\mu_{q,\ell}*\1_A\|_{\ell^2(\Z^d)}^{j}\,\|\mu_{q,\ell}*g\|_{\ell^2(\Z^d)}\,
 \bigl(\alpha\|\mu_{q,\ell}*\1_{[N]^d}\|_{\ell^2(\Z^d)}\bigr)^{K-1-j}\\
 &{}+N\|\theta\|_\infty q\ell\,(\alpha N^d)^{K-1}\|g\|_{\ell^1(\Z^d)}.
 \end{aligned}
 \end{aligned}
\]
We have $\|g\|_{\ell^1(\Z^d)}\leq2\alpha N^d$, and
$\ell\leq cN/t^{L+3}$ gives
\[
 N\|\theta\|_\infty q\ell
 \leq c\,t^{-(L+2)}
 \frac{N^2\|\theta\|_\infty}{1+N^2\|\theta\|_\infty}
 \leq c\,t^{-(L+2)}.
\]
For the product of the $\ell^2$ norms, we use
\eqref{eq:induction-A-second-moment}, \eqref{eq:induction-box-second-moment},
and \eqref{eq:induction-variance}. Since $q\ell/N\leq c$, these give
\[
 \begin{aligned}
 &\|\mu_{q,\ell}*\1_A\|_{\ell^2(\Z^d)}^{j}\,\|\mu_{q,\ell}*g\|_{\ell^2(\Z^d)}\,
 \bigl(\alpha\|\mu_{q,\ell}*\1_{[N]^d}\|_{\ell^2(\Z^d)}\bigr)^{K-1-j}\\
 &\qquad\ll_d \alpha^K N^{Kd/2}
 (N/\ell)^{jc/2}\bigl((N/\ell)^c-1+c\bigr)^{1/2}.
 \end{aligned}
\]
Since $j\leq K-1$ and $\ell<N$, dividing by $\alpha^KN^{Kd}$ gives
\begin{equation}\label{eq:induction-major-comparison}
 \begin{aligned}
 &\frac{\abs{\mathcal S_\beta(F_0,\ldots,F_{K-1})}}{\alpha^K N^{Kd}}\\
 &\quad\ll_{K,d}
 \mathfrak s_q(a)
 I_{2N}(\theta)
 (N/\ell)^{(K-1)c/2}\bigl((N/\ell)^c-1+c\bigr)^{1/2}
 +c\,t^{-(L+2)}.
 \end{aligned}
\end{equation}

To bound the integral of $\abs{\mathcal S_\beta(F_0,\ldots,F_{K-1})}$
over the major arcs, we sum \eqref{eq:induction-major-comparison} over
$q\leq Q$ and $a\bmod q$ with $(a,q)=1$, and integrate over
$\norm{\theta}_\infty\leq Q/(qR^2)$:
\begin{equation}\label{sumandint}
\sum_{q\leq Q}\sum_{\substack{a\bmod q\\(a,q)=1}}
 \mathfrak s_q(a)
 \int_{\norm{\theta}_\infty\leq Q/(qR^2)}
(N/\ell)^{(K-1)c/2}\bigl((N/\ell)^c-1+c\bigr)^{1/2}I_{2N}(\theta)
 \,d\theta.
\end{equation}
We need to compare $(N/\ell)^{(K-1)c/2}\bigl((N/\ell)^c-1+c\bigr)^{1/2}$ to $t^\gamma$ in order to apply Lemma \ref{lem:weighted-major-moment}. Fix a cutoff $H\geq1$. On the range $t\leq H$, the lower bound
$\ell\geq cN/(2t^{L+3})$ gives
\[
 \begin{aligned}
 (N/\ell)^{(K-1)c/2}\bigl((N/\ell)^c-1+c\bigr)^{1/2}\leq (2H^{L+3}/c)^{(K-1)c/2}
 \bigl((2H^{L+3}/c)^c-1+c\bigr)^{1/2}
 \longrightarrow 0,
 \end{aligned}
\]
as $c\to0$,
since $(2H^{L+3}/c)^c\to1$ for fixed $H$.
For $t>H$ and $c$ sufficiently small that $K(L+3)c\leq\gamma$, we have
\[
 \begin{aligned}
 &(N/\ell)^{(K-1)c/2}
 \bigl((N/\ell)^c-1+c\bigr)^{1/2}\ll_K (N/\ell)^{Kc/2}
 \leq(2/c)^{Kc/2}t^{K(L+3)c/2}
 \ll_K t^{\gamma/2}.
 \end{aligned}
\]
The constants here are uniform for sufficiently small $c$.
We also require $Q/R^2\leq1/(16N)$ for $N>N_0$.
Lemma~\ref{lem:weighted-major-moment} then gives
\begin{align*}
 &\sum_{q\leq Q}\sum_{\substack{a\bmod q\\(a,q)=1}}
 \mathfrak s_q(a)
 \int_{\norm{\theta}_\infty\leq Q/(qR^2)}
 t^\gamma
 I_{2N}(\theta)\,d\theta
 \ll_{K,d}N^{-2L}.
\end{align*}
Since $t\geq1$, the sum and integral of
\[
 \mathfrak s_q(a)I_{2N}(\theta)
 (N/\ell)^{(K-1)c/2}\bigl((N/\ell)^c-1+c\bigr)^{1/2}
\]
over $t\leq H$ are $o_{c\to0}(N^{-2L})$ for fixed $H$.
On $t>H$, we have $t^{\gamma/2}\leq H^{-\gamma/2}t^\gamma$,
so the sum and integral of the same expression over $t>H$ are
$O_{K,d}(H^{-\gamma/2}N^{-2L})$.
Choosing $H$ large and then $c$ small gives $o_{c\to0}(N^{-2L})$ in \eqref{sumandint}.

For the error $c\,t^{-(L+2)}$ in
\eqref{eq:induction-major-comparison}, there are at most $q^L$ numerators
modulo $q$, so
\[
 \begin{aligned}
 \sum_{q\leq Q}\sum_{\substack{a\bmod q\\(a,q)=1}}
 \int_{\norm{\theta}_\infty\leq Q/(qR^2)}c\,t^{-(L+2)}\,d\theta\leq c\sum_{q\geq1}q^{-2}
 \int_{\R^L}(1+N^2\norm{\theta}_\infty)^{-L-2}\,d\theta
 \ll_L cN^{-2L}.
 \end{aligned}
\]
Restoring the factor $\alpha^KN^{Kd}$ gives
\begin{equation}\label{eq:induction-major-error}
 \abs{\Lambda_N^{\mathfrak M_Q}(F_0,\ldots,F_{K-1})}
 \leq\int_{\mathfrak M_Q}\abs{\mathcal S_\beta(F_0,\ldots,F_{K-1})}\,d\beta
 \leq\varepsilon(c)\alpha^KN^{Kd-2L}.
\end{equation}
Here $\varepsilon(c)\to0$ as $c\to0$, uniformly in $j$ and $N>N_0(c,\Delta,d)$.
Combining this with the bound on the minor arc contribution yields \eqref{eq:induction-counting-comparison}, thus completing the proof. 
\end{proof}

\appendix
\counterwithout{equation}{section}
\setcounter{equation}{0}

\section{Application to a continuous variant of the problem}\label{app:continuous}

\begin{proof}[Proof of Corollary~\ref{cor:continuous}]
Let $C_0,c_0>0$ be the constants in
Theorem~\ref{thm:main}. Choose $D=D_{\Delta,d}\geq1$
sufficiently large and set
$N:=\left\lceil D\alpha^{-1/c_0}\right\rceil$.
Then $C_0N^{-c_0}<\alpha$ and
$N\leq2D\alpha^{-1/c_0}$.

For $u\in[0,1/N)^d$, define
\[
B_u:=
\left\{m\in\{0,\ldots,N-1\}^d:
u+\frac mN\in E\right\}.
\]
We partition $[0,1)^d$ by translates of $[0,1/N)^d$. Fubini's theorem gives
$\int_{[0,1/N)^d}\abs{B_u}\,du=\abs E$. Consequently, there is some $u$ such that $\abs{B_u}\geq\alpha N^d$.  By the choice of $N$ and
Theorem~\ref{thm:main}, the translate $B_u+(1,\ldots,1)\subseteq[N]^d$ contains a nondegenerate similar copy of $\Delta$. Undoing the translation, let
$y_0,\ldots,y_{K-1}\in B_u$ be the new vertices, and let $\rho>0$
be the similarity ratio.

Notice
\[
\rho^2
=\frac{\abs{y_i-y_j}^2}{\abs{v_i-v_j}^2}
\geq\frac1{\abs{v_i-v_j}^2}.
\]
Hence $\rho\geq c_\Delta>0$.  Mapping the similar copy of $\Delta$ into $E$ by
$m\mapsto u+m/N$ gives a similar copy with ratio $\lambda=\rho/N
\geq c_{\Delta,d}\alpha^{C_{\Delta,d}},$
for suitable constants $c_{\Delta,d},C_{\Delta,d}>0$.
Thus $x+\lambda U\Delta \subseteq E$, for some $x\in\R^d$ and $U\in\mathrm O(d)$.
\end{proof}

\section{Counting solutions to the matrix equation}
\label{app:matrix-count}

Expanding the trace in Lemma~\ref{lem:rectangular-moment} reduces the
required mean value estimate to counting bounded integer matrices $U$ and
$V$ satisfying $UV^{\mathsf T}=0$. Following is the precise estimate needed in the main argument.

\begin{lemma}[Integer matrix orthogonality]\label{lem:matrix-orthogonality}
Let $s,t,n$ be positive integers with $n\geq2\max(s,t)$ and $n>s+t$.
Then, for every positive integer $R$,
\[
\#\left\{(U,V):
\begin{array}{l}
  U\in ([-R,R]\cap\Z)^{s\times n},\\
  V\in ([-R,R]\cap\Z)^{t\times n},\\
  UV^{\mathsf T}=0
\end{array}
\right\}
\ll_{n,s,t}R^{n(s+t)-2st}.
\]
\end{lemma}

In our application, $n=K+1$, $s=\abs S$ and $t=\abs T$. Thus
$n>s+t$, as $\abs S+\abs T=K$, while
$n\geq2\max(s,t)$ follows from the balancedness of the cut $S\mid T$.
We now develop the lattice terminology needed for the proof.

\subsection{Lattice preliminaries}

The definitions and elementary determinant facts below are standard;
see \cite[Chapter~I]{CasselsGeometryNumbers}. We follow
\cite[\S1]{KimFlags} for the determinant convention and for the notation
$X_n$, $\mu_n$, $\mathrm{Gr}(L,d)$, and rational flags.

A lattice in $\R^n$ is a subgroup of the form
$\Gamma=\Z\gamma_1+\cdots+\Z\gamma_d,$
where $\gamma_1,\ldots,\gamma_d\in\R^n$ are linearly independent.
These vectors form a basis of $\Gamma$, and its rank is $d$. The real span of $\Gamma$ is
$\operatorname{span}_{\R}\Gamma := \R\gamma_1+\cdots+\R\gamma_d$.
A sublattice of a lattice $L$ is a subgroup
$\Gamma\subseteq L$; every such subgroup is itself a lattice.

The Gram matrix of the basis $\gamma_1,\ldots,\gamma_d$ is
$G(\gamma_1,\ldots,\gamma_d) := \bigl(\langle\gamma_i,\gamma_j\rangle\bigr)_{i,j=1}^d$.
Following Kim, define the determinant of $\Gamma$ by
\[
\det\Gamma
:=
\lVert\gamma_1\wedge\cdots\wedge\gamma_d\rVert
=
\bigl(\det G(\gamma_1,\ldots,\gamma_d)\bigr)^{1/2}.
\]
Any two bases of $\Gamma$ differ by a matrix
$A\in\operatorname{GL}_d(\Z)$. Under this change of basis, the Gram
matrix $G$ is replaced by $AGA^{\mathsf T}$, whose determinant equals
$\det G$. Thus $\det\Gamma$ does not depend on the chosen basis.

If $\Gamma\subseteq\Lambda$ has the same rank, then $\Gamma$ has
finite index in $\Lambda$. Denote the number of cosets of $\Gamma$ in $\Lambda$ by  $[\Lambda:\Gamma]$. Then we have
$\det\Gamma=[\Lambda:\Gamma]\det\Lambda$. Therefore, enlarging a lattice
without changing its rank can only decrease its determinant. Moreover,
if $\Gamma\subseteq\Z^n$, then the Gram matrix of any basis of
$\Gamma$ has integer entries and a positive integer determinant, so
$\det\Gamma\geq1$.

A lattice in $\R^n$ is full-rank if it has rank $n$, and it is
unimodular if it is full-rank and its determinant equals one. If
$g\in\operatorname{GL}_n(\R)$, then $\Z^ng:=\{zg:z\in\Z^n\}$
is a full-rank lattice of determinant $\abs{\det g}$, where vectors
are regarded as row vectors. In particular, $\Z^ng$ is unimodular
when $g\in\operatorname{SL}(n,\R)$.

Following Kim, let
$X_n := \operatorname{SL}(n,\Z)\backslash\operatorname{SL}(n,\R).$
The coset of $g\in\operatorname{SL}(n,\R)$ corresponds to the
unimodular row lattice $\Z^ng$. Left multiplication by an element of
$\operatorname{SL}(n,\Z)$ merely changes the chosen lattice basis, and
every unimodular lattice has such a basis matrix. Thus $X_n$ is the
space of unimodular lattices in $\R^n$. We equip it with the invariant
probability measure $\mu_n$ induced by Haar measure on
$\operatorname{SL}(n,\R)$.

A sublattice $\Gamma\subseteq L$ is primitive in
$L$ if
$\Gamma = L\cap\operatorname{span}_{\R}\Gamma.$
For an arbitrary sublattice $\Gamma\subseteq L$, the lattice
$L\cap\operatorname{span}_{\R}\Gamma$
is called the saturation of $\Gamma$ in $L$. It is primitive in
$L$ and has the same rank as $\Gamma$.

For $L\in X_n$ and $1\leq d<n$, let $\mathrm{Gr}(L,d)$ denote
the set of primitive rank-$d$ sublattices of $L$.

For a sublattice $\Gamma\subseteq\Z^n$, define its integral
orthogonal complement by
\[
\Gamma^{\perp_{\Z}}
:=
\left\{
x\in\Z^n:
\langle x,\gamma\rangle=0
\text{ for every }\gamma\in\Gamma
\right\}
=
\bigl(\operatorname{span}_{\R}\Gamma\bigr)^\perp\cap\Z^n.
\]
If $\Gamma\subseteq\Z^n$ is primitive of rank $d$, standard
lattice duality \cite[\S2, especially Lemma~1 and its
corollary]{SchmidtPointLattices} shows that $\Gamma^{\perp_{\Z}}$ is
primitive of rank $n-d$ and that
$\det(\Gamma^{\perp_{\Z}})=\det\Gamma$. Taking the integral
orthogonal complement twice gives $(\Gamma^{\perp_{\Z}})^{\perp_{\Z}}=\Gamma$.

For a rank-$d$ lattice $\Gamma$ and $1\leq j\leq d$, its
$j$-th successive minimum $\lambda_j(\Gamma)$ is the least
$\lambda>0$ for which $\Gamma$ contains $j$ linearly independent
vectors of Euclidean norm at most $\lambda$. Thus
$\lambda_1(\Gamma)\leq\cdots\leq\lambda_d(\Gamma).$
A standard consequence of the lattice-point estimate in
\cite[\S5, Lemma~2]{SchmidtPointLattices} is
\begin{equation}\label{eq:lattice-point-successive-minima}
  \#(\Gamma\cap[-R,R]^n)
  \ll_n
  \prod_{j=1}^d\left(1+\frac{R}{\lambda_j(\Gamma)}\right).
\end{equation}
If $\lambda_d(\Gamma)\ll_n R$, then
$1+R/\lambda_j(\Gamma)\ll_n R/\lambda_j(\Gamma)$ for every
$1\leq j\leq d$. Minkowski's second theorem
\cite[Chapter~VIII, Theorem~V]{CasselsGeometryNumbers} gives
$\prod_{j=1}^d\lambda_j(\Gamma)\asymp_n\det\Gamma$, so
\eqref{eq:lattice-point-successive-minima} yields
\begin{equation}\label{eq:lattice-point-determinant}
  \#(\Gamma\cap[-R,R]^n)
  \ll_n\frac{R^d}{\det\Gamma}.
\end{equation}

Finally, we fix Kim's notation for the flags appearing below. Let
$V(i):=\pi^{i/2}/\Gamma(1+i/2)$
be the volume of the unit ball in $\R^i$, and adopt the convention
$\zeta(1)=1$. For $1\leq d<n$, define
\[
a(n,d)
:=
\frac1n\binom nd
\prod_{i=1}^d
\frac{V(n-i+1)\zeta(i)}{V(i)\zeta(n-i+1)}.
\]
Let $\mathfrak d=(d_1,\ldots,d_k)$ with
$0<d_1<\cdots<d_k<n$, and set $d_0=0$ and $d_{k+1}=n$.
Kim calls a chain
$A_1\subseteq\cdots\subseteq A_k\subseteq L$
a flag of type $\mathfrak d$ rational with respect to $L$ if
$A_i\in\mathrm{Gr}(L,d_i)$ for every $1\leq i\leq k$. He defines
\[
a(n,\mathfrak d)
:=
a(n,d_1)
\prod_{i=1}^{k-1}
\frac{n-d_{i-1}}{d_{i+1}-d_{i-1}}
a(n-d_i,d_{i+1}-d_i).
\]

\subsection{Kim's mean value formula for primitive lattice flags}

Now we record the only non-elementary input.

\begin{theorem}[Kim's mean value formula for rational flags
{\cite[Theorem~5]{KimFlags}}]\label{thm:kim-flag-mean}
Let $H_1,\ldots,H_k\geq0$. The $\mu_n$-average over $L\in X_n$ of the count of flags $A_1\subseteq\cdots\subseteq A_k$ of type $\mathfrak d$ that are rational with respect to $L$ and satisfy $\det A_i\leq H_i$, for every $1\leq i\leq k$, is $a(n,\mathfrak d) \prod_{i=1}^k H_i^{d_{i+1}-d_{i-1}}.$

\end{theorem}

We transfer the case $k=2$ from a random unimodular lattice to the
fixed lattice $\Z^n$.

\begin{lemma}[Two-height flag count]\label{lem:two-height-flag-count}
Fix positive integers $n,d_1,d_2$ with
$0<d_1<d_2<n.$
Uniformly for $H_1,H_2\geq1$,
\[
\#\left\{(A_1,A_2):
\begin{array}{l}
  A_1\in\mathrm{Gr}(\Z^n,d_1),\\
  A_2\in\mathrm{Gr}(\Z^n,d_2),\\
  A_1\subseteq A_2,\\
  \det A_1\leq H_1,\quad \det A_2\leq H_2
\end{array}
\right\}
\ll_{n,d_1,d_2}H_1^{d_2}H_2^{n-d_1}.
\]
\end{lemma}

\begin{proof}
Take $k=2$ and $\mathfrak d=(d_1,d_2)$ in
Theorem~\ref{thm:kim-flag-mean}. Since $d_0=0$ and $d_3=n$, it gives
\begin{equation}\label{eq:kim-two-step-mean}
  \begin{aligned}
    \int_{X_n}
    \#\left\{(A_1,A_2):
    \begin{array}{l}
      A_1\in\mathrm{Gr}(L,d_1),\\
      A_2\in\mathrm{Gr}(L,d_2),\\
      A_1\subseteq A_2,\\
      \det A_1\leq H_1,\quad \det A_2\leq H_2
    \end{array}
    \right\}
    \,d\mu_n(L)&=
    a(n,\mathfrak d)
    H_1^{d_2-d_0}H_2^{d_3-d_1}\\
    &=
    a(n,\mathfrak d)H_1^{d_2}H_2^{n-d_1}.
  \end{aligned}
\end{equation}

It remains to pass from this average to the particular lattice $\Z^n$.
Let
$\pi:\operatorname{SL}(n,\R)\longrightarrow X_n$
be the quotient map. Choose a relatively compact open neighborhood
$W$ of the identity on which $\pi$ is injective, and put
$\mathcal U=\pi(W)$. Then $\mu_n(\mathcal U)>0$. For
$A\subseteq\Z^n$ and $g\in\operatorname{SL}(n,\R)$, write
$Ag:=\{ag:a\in A\}$. For $1\leq j<n$, the relative compactness of
$W$ gives a constant
$c_j<\infty$ such that
$\det(Ag) \leq c_j\det A$
for every $A\in\mathrm{Gr}(\Z^n,j)$ and every $g\in W$.
Right multiplication by $g$ maps $\mathrm{Gr}(\Z^n,j)$
bijectively onto $\mathrm{Gr}(\Z^ng,j)$ and preserves inclusions.
Hence, for every $g\in W$,
\begin{align*}
  &\#\left\{(A_1,A_2):
  \begin{array}{l}
    A_1\in\mathrm{Gr}(\Z^n,d_1),\\
    A_2\in\mathrm{Gr}(\Z^n,d_2),\\
    A_1\subseteq A_2,\\
    \det A_1\leq H_1,\quad \det A_2\leq H_2
  \end{array}
  \right\}\leq
  \#\left\{(A_1,A_2):
  \begin{array}{l}
    A_1\in\mathrm{Gr}(\Z^ng,d_1),\\
    A_2\in\mathrm{Gr}(\Z^ng,d_2),\\
    A_1\subseteq A_2,\\
    \det A_1\leq c_{d_1}H_1,\\
    \det A_2\leq c_{d_2}H_2
  \end{array}
  \right\}.
\end{align*}
Integrating this inequality over $\mathcal U$ and applying
\eqref{eq:kim-two-step-mean} with
$(H_1,H_2)$ replaced by
$(c_{d_1}H_1,c_{d_2}H_2)$, we obtain
\begin{align*}
  &\mu_n(\mathcal U)
  \#\left\{(A_1,A_2):
  \begin{array}{l}
    A_1\in\mathrm{Gr}(\Z^n,d_1),\\
    A_2\in\mathrm{Gr}(\Z^n,d_2),\\
    A_1\subseteq A_2,\\
    \det A_1\leq H_1,\quad \det A_2\leq H_2
  \end{array}
  \right\}\leq
  a(n,\mathfrak d)
  (c_{d_1}H_1)^{d_2}(c_{d_2}H_2)^{n-d_1}.
\end{align*}
Since $W$ is fixed in terms of $n$, this proves the result.
\end{proof}

\subsection{Proof of the matrix count}

\begin{proof}[Proof of Lemma~\ref{lem:matrix-orthogonality}]
We decompose the count according to the ranks of $U$ and $V$.
Write $d_1=\rank U$ and $d_2=n-\rank V$. If $d_1=0$, then $U=0$, and there are $O(R^{nt})$ choices for $V$. This is within the stated bound because
\[
n(s+t)-2st-nt=s(n-2t)\geq0.
\]
In the case that $d_2=n$, which implies $V=0$, we make the same argument. Therefore we may assume $1\leq d_1\leq s$ and $n-t\leq d_2\leq n-1$. As $d_1+(n-d_2)\leq s+t<n,$ we have $d_1<d_2$.

The rows of $U$ generate a sublattice of $\Z^n$. The intersection of the real span of those rows with $\Z^n$ is its saturation.
Denote this saturated row lattice and its analogue for $V$ by
\[
\Gamma_U
:=\operatorname{rowspan}_{\R}(U)\cap\Z^n,
\qquad
\Gamma_V
:=\operatorname{rowspan}_{\R}(V)\cap\Z^n.
\]
By construction, $\Gamma_U$ and $\Gamma_V$ are primitive sublattices of $\Z^n$ of ranks
$d_1$ and $n-d_2$ respectively. The equation
$UV^{\mathsf T}=0$ implies that every row of $U$ is orthogonal to every row of $V$. Define $\Lambda_V:=\Gamma_V^{\perp_{\Z}}$. Then it is primitive of rank $d_2$, due to lattice duality. Thus
$\Gamma_U\subseteq\Lambda_V\subseteq\Z^n$ is a flag of type
$(d_1,d_2)$ and rational with respect to $\Z^n$. Moreover,
$\det\Lambda_V=\det\Gamma_V$ and $\Gamma_V=\Lambda_V^{\perp_{\Z}}$.
The latter identity shows that the flag determines the pair
$(\Gamma_U,\Gamma_V)$.

We next count the matrices giving a fixed pair of row lattices. Since $U$ has rank $d_1$, it has $d_1$ linearly independent rows, each of Euclidean norm $O_n(R)$. Hence
$\lambda_{d_1}(\Gamma_U)\ll_n R.$
Let $\Gamma'$ be the lattice generated by these $d_1$ rows.
Hadamard's inequality gives $\det\Gamma'\ll_n R^{d_1}$.
Since $\Gamma_U$ is the saturation of $\Gamma'$, we have
$\det\Gamma_U=\det\Gamma'/[\Gamma_U:\Gamma']\leq\det\Gamma'$.
As $\Gamma_U\subseteq\Z^n$, it follows that
$1\leq\det\Gamma_U\ll_n R^{d_1}$.
Applying the same argument to $V$ gives
$\lambda_{n-d_2}(\Gamma_V)\ll_n R$ and
$1\leq\det\Gamma_V\ll_n R^{n-d_2}$.
It follows from \eqref{eq:lattice-point-determinant} that, once
$\Gamma_U$ and $\Gamma_V$ are fixed, the numbers of possible
matrices are bounded by
\begin{equation}\label{eq:fixed-row-lattice-counts}
  \#\{U\}
  \ll_{n,s}\left(
  \frac{R^{d_1}}{\det\Gamma_U}
  \right)^s,
  \qquad
  \#\{V\}
  \ll_{n,t}\left(
  \frac{R^{n-d_2}}{\det\Gamma_V}
  \right)^t.
\end{equation}

Let $N_{d_1,d_2}(R)$ be the contribution from matrices with these fixed ranks. Choose nonnegative integers $J_1, J_2$ such that all possible determinants are covered by the dyadic ranges $2^j\leq\det\Gamma_U<2^{j+1}$ and $2^\ell\leq\det\Gamma_V<2^{\ell+1}$, with integers $j,\ell$ in $0\leq j\leq J_1$, and $0\leq\ell\leq J_2$. We may assume $2^{J_1}\asymp_n R^{d_1}$ and $2^{J_2}\asymp_n R^{n-d_2}$.
For fixed $j$ and $\ell$, the orthogonal complement identity and
Lemma~\ref{lem:two-height-flag-count}, with heights
$2^{j+1}$ and $2^{\ell+1}$, show that the number of possible pairs
$(\Gamma_U,\Gamma_V)$ is $\ll_{n,d_1,d_2}2^{jd_2}2^{\ell(n-d_1)}.$

For each such pair, \eqref{eq:fixed-row-lattice-counts} bounds the
number of matrices by $\ll_{n,s,t}
R^{d_1s+(n-d_2)t}2^{-js}2^{-\ell t}.$
Thus
\begin{align*}
  N_{d_1,d_2}(R)
  &\ll_{n,s,t}
  R^{d_1s+(n-d_2)t}
  \sum_{j=0}^{J_1}\sum_{\ell=0}^{J_2}
  2^{j(d_2-s)}2^{\ell(n-d_1-t)}
  \\
  &=
  R^{d_1s+(n-d_2)t}
  \left(\sum_{j=0}^{J_1}2^{j(d_2-s)}\right)
  \left(\sum_{\ell=0}^{J_2}2^{\ell(n-d_1-t)}\right)
  \\
  &\ll_{n,s,t}
  R^{d_1s+(n-d_2)t}
  2^{J_1(d_2-s)}2^{J_2(n-d_1-t)}
  \\
  &\ll_{n,s,t}
  R^{d_1s+(n-d_2)t}
  R^{d_1(d_2-s)}
  R^{(n-d_2)(n-d_1-t)}
  \\
  &=
  R^{d_1s+(n-d_2)t+d_1d_2-d_1s
    +(n-d_2)(n-d_1)-(n-d_2)t}
  \\
  &=
  R^{d_1d_2+(n-d_2)(n-d_1)}
  \\
  &=
  R^{n^2-n(d_1+d_2)+2d_1d_2}.
\end{align*}

It remains to maximize the exponent
$E(d_1,d_2):=n^2-n(d_1+d_2)+2d_1d_2$
over $1\leq d_1\leq s$ and $n-t\leq d_2\leq n-1$. We have
\[
E(d_1+1,d_2)-E(d_1,d_2)
=2d_2-n
\geq n-2t
\geq0
\]
and
\[
E(d_1,d_2-1)-E(d_1,d_2)
=n-2d_1
\geq n-2s
\geq0.
\]
Thus $E(d_1,d_2)$ is maximized at $d_1=s$ and $d_2=n-t$, where
$E(s,n-t)=n(s+t)-2st$.
Summing over the $O_n(1)$ choices for $(d_1,d_2)$ completes the proof.
\end{proof}

\providecommand{\bysame}{\leavevmode\hbox to3em{\hrulefill}\thinspace}
\providecommand{\MR}{\relax\ifhmode\unskip\space\fi MR }
\providecommand{\MRhref}[2]{%
\href{http://www.ams.org/mathscinet-getitem?mr=#1}{#2}
}
\providecommand{\href}[2]{#2}

\end{document}